\documentclass[11pt]{amsart}
\usepackage{hyperref} 

\usepackage{amsthm}
\usepackage{amsmath,amssymb}
\usepackage{amsxtra}
\usepackage[latin1]{inputenc}
\usepackage[T1]{fontenc}
\usepackage{amssymb,latexsym,array}
\usepackage[UKenglish]{babel}
\usepackage[all]{xy}
\usepackage{algorithmicx}
\usepackage[ruled]{algorithm}
\usepackage{algpseudocode}
\usepackage{algc}
\newenvironment{alginc}[1][pseudocode]{\medskip\algsetlanguage{#1}\begin{algorithmic}[0]}{\end{algorithmic}\medskip}
\usepackage{tikz}
\usetikzlibrary{matrix,arrows}
\usepackage{geometry}
\newcommand{\aTop}[2]{\begin{array}{c}{#1}\\{#2}\end{array}}
\newcommand{\C}{\mathbb C}
\newcommand{\N}{\mathbb N}
\newcommand{\R}{\mathbb R}

\newcommand{\Z}{\mathbb Z}

\newcommand{\rationals}{\mathbb Q}
\newcommand{\Hy}{\mathcal{H}}
\newcommand{\ringO}{\mathcal{O}}
\newcommand{\less}{\medspace \medspace < \medspace \medspace}
\newcommand{\mat}{\begin{pmatrix}a & b\\c & d\end{pmatrix}}
\newcommand{\suchthat}{ \medspace | \medspace}
\newcommand{\Real}{{\rm Re}}
\newcommand{\Imag}{{\rm Im}}

\newcommand*{\Homol}{\operatorname{H}}
\newcommand{\T}{{\mathbb{T}}}

\renewcommand{\leq}{\leqslant}
\renewcommand{\geq}{\geqslant}

\newtheorem{theorem}{\bfseries Theorem}
\newtheorem{proposition}[theorem]{\bfseries Proposition}

\newtheorem{lemma}[theorem]{\bfseries Lemma}

\theoremstyle{definition}
\newtheorem{definition}[theorem]{\bfseries Definition}
\newtheorem{notation}[theorem]{\bfseries Notation}
\newtheorem{observation}[theorem]{\bfseries Observation}
\newtheorem{remark}[theorem]{\bfseries Remark}

\begin{document}

\title[Higher torsion in the Abelianization of the full Bianchi groups]{Higher torsion in the Abelianization \\ of the full Bianchi groups}
\author{Alexander D. Rahm}
\email{Alexander.Rahm@nuigalway.ie}
\urladdr{http://www.maths.nuigalway.ie/\char126rahm/}
\address{National University of Ireland at Galway, Department of Mathematics}
\thanks{Funded by the Irish Research Council.}  
\date{\today}

\begin{abstract}
Denote by $\rationals(\sqrt{-m})$, with $m$ a square-free positive integer, an imaginary quadratic number field, and by $\ringO_{-m}$ its ring of integers.
The \emph{\mbox{Bianchi} groups} are the groups $\mathrm{SL_2}(\ringO_{-m})$.
In the literature, there has been so far no example of $p$-torsion in the integral homology of the full Bianchi groups,
 for $p$ a prime greater than the order of elements of finite order in the Bianchi group, which is at most $6$.

However, extending the scope of the computations,
we can observe examples of torsion in the integral homology of the quotient space,
at prime numbers as high as for instance $p = 80737$ at the discriminant $-1747$.
\end{abstract}

\maketitle


\section{Introduction} 
The \mbox{Bianchi} groups $\Gamma := \mathrm{SL_2}(\ringO_{-m})$
 may be considered as a key to the study of a larger class of groups,
 the \emph{Kleinian} groups, which date back to work of Henri Poincar\'e~\cite{Poincare}.
In fact, each non-co-compact arithmetic Kleinian group is commensurable with some \mbox{Bianchi} group~\cite{MaclachlanReid}.
A wealth of information on the \mbox{Bianchi} groups can be found in the monographs \cite{Fine},
 \cite{ElstrodtGrunewaldMennicke}, \cite{MaclachlanReid}.
In the literature, there has been so far no example of $p$-torsion in the integral homology of the full Bianchi groups,
 for $p$ a prime greater than the order of elements of finite order in the Bianchi group
(a recent survey of relevant calculations has been given in \cite{Sengun_survey}).
In fact, the numerical studies that have been made so far,
 were carried out in the range where the quotient space of hyperbolic $3$-space $\Hy$ by the Bianchi group 
 is often homotopy equivalent to a wedge sum of $2$-spheres, $2$-tori and M\"obius bands \cite{Vogtmann}.

We make use of Serre's decomposition \cite{Serre} of the homology group $\Homol_1(_\Gamma\backslash \Hy; \thinspace \Z)$
 into the direct sum of the free Abelian group with one generator for each element of the class group of~$\ringO_{-m}$
and the group $\Homol_1^{\rm \overline{cusp}}(_\Gamma\backslash \Hy; \thinspace \Z)$
 computed in figures~\ref{small absolute values} and~\ref{greater absolute values}.
The first compuations of $\Homol_1(\Gamma; \thinspace \Z) \supset \Homol_1(_\Gamma\backslash \Hy; \thinspace \Z)$
by Swan~\cite{Swan} were on a range of Bianchi groups with vanishing cusp-complementary homology
 $\Homol_1^{\rm \overline{cusp}}(_\Gamma\backslash \Hy; \thinspace \Z)$. The first example where
$\Homol_1^{\rm \overline{cusp}}(_\Gamma\backslash \Hy; \thinspace \Z)$ is non-zero,
occurred in an unpublished calculation of \mbox{Mennicke}.
Swan's manual computations of group presentations have been extended on the computer by Riley~\cite{Riley};
and later Vogtmann~\cite{Vogtmann} and Scheutzow~\cite{Scheutzow} systematically computed 
$\Homol_1^{\rm \overline{cusp}}(_\Gamma\backslash \Hy; \thinspace \rationals)$ for a large range of Bianchi groups.
But they were still in in the range where $_\Gamma\backslash \Hy$ admits no homological torsion.
Aranes~\cite{Aranes} has computed cell complexes for the Bianchi groups for all $m \leq 100$,
 and Yasaki~\cite{Yasaki} has obtained $\mathrm{GL_2}(\ringO_{-m})$-cell complexes (with the Vorono\"i model)
for the same range as well as all cases where $\ringO$ is of class number $1$ or $2$.
This includes two cases, $m = 74$ and $m = 86$, where some $2$-torsion appears in 
$\Homol_1^{\rm \overline{cusp}}(_\Gamma\backslash \Hy; \thinspace \Z)$,
 but the latter two authors have not yet provided homology computations.
When the absolute value of the discriminant gets greater,
 torsion in the integral homology of the quotient space appears (see figure~\ref{greater absolute values})
at prime numbers as high as for instance $80737$ at the discriminant $-1747$, 
whereas the order  of elements of finite order in $\mathrm{SL_2}(\ringO_{-m})$ is at most~$6$.
A growth of the torsion in the Abelianization of the Bianchi groups with respect to the covolume can be observed,
which is in concordance with the predictions of \cite{BergeronVenkatesh}.
We can also observe that the occurring torsion subgroups are quite likely to occur as squares,
but this is no general principle, because the discriminant $-431$ produces a counterexample to this phenomenon.

In order to obtain the results of figures~\ref{small absolute values} and~\ref{greater absolute values},
in section~\ref{realization}
 we fill out Swan's concept \cite{Swan} and elaborate algorithms to compute a fundamental polyhedron for the action of the Bianchi groups on
hyperbolic $3$-space.
Other algorithms based on the same concept have independently been implemented by Cremona~\cite{Cremona}
 for the five cases where~$\ringO_{-m}$ is Euclidean, and by his students Whitley~\cite{Whitley}
 for the non-Euclidean principal ideal domain cases, Bygott~\cite{Bygott} for a case of class number 2 and Lingham (\cite{Lingham},
 used in~\cite{CremonaLingham}) for some cases of class number 3; and finally Aran\'es~\cite{Aranes} for arbitrary class numbers.
The algorithms presented in subsection \ref{realization} come with an implementation~\cite{BianchiGP}
 for all Bianchi groups; and we make explicit use of the cell complexes it produces.
The provided implementation~\cite{BianchiGP} has been validated by the project PLUME of the CNRS,
and is subject to the certificate C3I of the GENCI and the CPU. 
Other results obtained with the employed implementation are described in~\cite{RahmNoteAuxCRAS} and \cite{RahmSengun}.
On the computing clusters of the Weizmann Institute of Science,
 this implementation has been applied to establish a database of cell complexes for over $180$ Bianchi groups,
 using over fifty processor-months.
This database includes all the cases of ideal class numbers $3$ and $5$, most of the cases of ideal class number $4$ 
and all of the cases of discriminant absolute value bounded by the number $500$.

A computational advantage is the shortcut that we obtain in section~\ref{connection}
 by linking the Borel--Serre compactification of the quotient space with Fl\"oge's compactification
in a long exact sequence, based on the recent paper~\cite{On a question of Serre}.
Fl\"oge's compactification admits a computationally easier cell structure,
 and we can explicitly calculate the equivariant Leray--Serre spectral sequence associated to it.
In section~\ref{SpecSeq}, we describe how to assemble the homology of the Borel--Serre compactified quotient space 
and the Farrell cohomology of a Bianchi group to its full group homology with trivial $\Z$--coefficients.
Here, we divide by the center of $\mathrm{SL_2}(\ringO_{-m})$, consisting of plus and minus the identity matrix,
yielding $\mathrm{PSL_2}(\ringO_{-m})$.
As the center of $\mathrm{SL_2}(\ringO_{-m})$ is the kernel of its action on hyperbolic $3$-space,
this does not change the quotient space.
And for $\Gamma := \mathrm{PSL_2}(\ringO_{-m})$,
 general formulae for its Farrell cohomology have been given~\cite{accessing Farrell cohomology}
(based on~\cite{homological torsion}).

\subsection{Organization of the paper} 
We print the isomorphism types of 
$\Homol_1^{\rm \overline{cusp}}(_\Gamma\backslash \Hy; \thinspace \Z)$ that were obtained
in figures~\ref{small absolute values} and~\ref{greater absolute values}.
The homology group $\Homol_1(_\Gamma\backslash \Hy; \thinspace \Z)$
is a direct sum of the former and the free Abelian group with rank the cardinality of the class group of~$\ringO_{-m}$,
which we also print. 
There is an inclusion of $\Homol_1(_\Gamma\backslash \Hy; \thinspace \Z)$ into the group homology $\Homol_1(\Gamma; \thinspace \Z)$;
and the latter group homology is a quotient of the direct sum of $\Homol_1(_\Gamma\backslash \Hy; \thinspace \Z)$
and the Farrell supplement that has been computed and printed in a separate column.
In section~\ref{The Bianchi fundamental polyhedron}, we define the Bianchi fundamental polyhedron,
which induces our cell structure on $_\Gamma\backslash \Hy$.
 We use it in section~\ref{Floege cellular complex} to obtain the Fl\"oge cellular complex, 
which we connect in section~\ref{connection} to the Borel--Serre compactification of $_\Gamma\backslash \Hy$.
Then we proceed to $\Homol_1(\mathrm{PSL_2}(\ringO_{-m}); \thinspace \Z)$ in section~\ref{SpecSeq},
describe Swan's concept in section~\ref{Swan's concept} and its realization in section~\ref{realization}.

\begin{figure}\scriptsize
\mbox{
 $\begin{array}
  {|c|c|c|c|c|}
\hline 
\Delta & m & { \aTop{\rm class}{\rm group}}
 & \Homol_1^{\rm \overline{cusp}}
 & \aTop{\rm Farrell}{\rm supplement} \\
\hline 
 -7 & 7 & \{1\} & 0  & \Z/2  \\
 -8 & 2 & \{1\} & 0  & \Z/2 \oplus \Z/3 \\
 -11 & 11 & \{1\} & 0  & \Z/3 \\
 -15 & 15 & \Z/2 & 0  & \Z/2 \oplus \Z/3 \\
 -19 & 19 & \{1\} & 0  & 0  \\
 -20 & 5 & \Z/2 & 0  & (\Z/2)^{2} \oplus \Z/3 \\
 -23 & 23 & \Z/3 & 0  & \Z/2 \oplus \Z/3 \\
 -24 & 6 & \Z/2 & 0  & \Z/2 \oplus \Z/3 \\
 -31 & 31 & \Z/3 & 0  & \Z/2  \\
 -35 & 35 & \Z/2 & \Z  & \Z/2 \oplus \Z/3 \\
 -39 & 39 & \Z/4 & 0  & \Z/2 \oplus \Z/3 \\
 -40 & 10 & \Z/2 & \Z  & (\Z/2)^{2} \oplus \Z/3 \\
 -43 & 43 & \{1\} & \Z  & 0  \\
 -47 & 47 & \Z/5 & 0  & \Z/2 \oplus \Z/3 \\
 -51 & 51 & \Z/2 & \Z  & \Z/3 \\
 -52 & 13 & \Z/2 & \Z  & (\Z/2)^{2}  \\
 -55 & 55 & \Z/4 & \Z  & \Z/2 \oplus \Z/3 \\
 -56 & 14 & \Z/4 & \Z  & (\Z/2)^{2} \oplus \Z/3 \\
 -59 & 59 & \Z/3 & \Z  & \Z/3 \\
 -67 & 67 & \{1\} & \Z^{2}  & 0  \\
 -68 & 17 & \Z/4 & \Z  & (\Z/2)^{2} \oplus \Z/3 \\
 -71 & 71 & \Z/7 & 0  & \Z/2 \oplus \Z/3 \\
 -79 & 79 & \Z/5 & \Z  & (\Z/2)^{3}  \\
 -83 & 83 & \Z/3 & \Z^{2}  & \Z/3 \\
 -84 & 21 & \Z/2 \times \Z/2 & \Z^{3}  & (\Z/2)^{3} \oplus (\Z/3)^{2} \\
 -87 & 87 & \Z/6 & \Z^{2}  & \Z/2 \oplus \Z/3 \\
 -88 & 22 & \Z/2 & \Z^{3}  & \Z/2 \oplus \Z/3 \\
 -91 & 91 & \Z/2 & \Z^{3}  & \Z/2  \\
 -95 & 95 & \Z/8 & \Z  & \Z/2 \oplus \Z/3 \\
 -103 & 103 & \Z/5 & \Z^{2}  & \Z/2  \\
 -104 & 26 & \Z/6 & \Z^{2}  & (\Z/2)^{2} \oplus (\Z/3)^{2} \\
 -107 & 107 & \Z/3 & \Z^{3}  & (\Z/3)^{3} \\
 -111 & 111 & \Z/8 & \Z^{2}  & \Z/2 \oplus \Z/3 \\
 -115 & 115 & \Z/2 & \Z^{5}  & \Z/2 \oplus \Z/3 \\
 -116 & 29 & \Z/6 & \Z^{3}  & (\Z/2)^{2} \oplus \Z/3 \\
 -119 & 119 & \Z/10 & \Z  & (\Z/2)^{2} \oplus \Z/3 \\
 -120 & 30 & \Z/2 \times \Z/2 & \Z^{6}  & (\Z/2)^{3} \oplus (\Z/3)^{3} \\
 -123 & 123 & \Z/2 & \Z^{5}  & \Z/3 \\
 -127 & 127 & \Z/5 & \Z^{3}  & \Z/2  \\
 -131 & 131 & \Z/5 & \Z^{3}  & \Z/3 \\
 -132 & 33 & \Z/2 \times \Z/2 & \Z^{6}  & (\Z/2)^{3} \oplus (\Z/3)^{4} \\
 -136 & 34 & \Z/4 & \Z^{4}  & (\Z/2)^{4} \oplus \Z/3 \\
 -139 & 139 & \Z/3 & \Z^{4}  & 0  \\
 -143 & 143 & \Z/10 & \Z^{2}  & \Z/2 \oplus (\Z/3)^{2} \\
 -148 & 37 & \Z/2 & \Z^{6}  & (\Z/2)^{4}  \\
 -151 & 151 & \Z/7 & \Z^{3}  & \Z/2  \\
 -152 & 38 & \Z/6 & \Z^{4}  & \Z/2 \oplus \Z/3 \\
 -155 & 155 & \Z/4 & \Z^{6}  & \Z/2 \oplus \Z/3 \\
 -159 & 159 & \Z/10 & \Z^{4}  & \Z/2 \oplus \Z/3 \\
 -163 & 163 & \{1\} & \Z^{6}  & 0  \\
 -164 & 41 & \Z/8 & \Z^{4}  & (\Z/2)^{2} \oplus \Z/3 \\
 -167 & 167 & \Z/11 & \Z^{2}  & \Z/2 \oplus \Z/3 \\
 -168 & 42 & \Z/2 \times \Z/2 & \Z^{9}  & (\Z/2)^{3} \oplus (\Z/3)^{2} \\
 -179 & 179 & \Z/5 & \Z^{5}  & \Z/3 \\
 -183 & 183 & \Z/8 & \Z^{6}  & \Z/2 \oplus \Z/3 \\
 -184 & 46 & \Z/4 & \Z^{7}  & (\Z/2)^{2} \oplus \Z/3 \\
 -187 & 187 & \Z/2 & \Z^{7}  & \Z/3 \\
 -191 & 191 & \Z/13 & \Z^{2}  & \Z/2 \oplus \Z/3 \\
 -195 & 195 & \Z/2 \times \Z/2 & \Z^{11}  & (\Z/2)^{2} \oplus (\Z/3)^{2} \\
 -199 & 199 & \Z/9 & \Z^{4}  & \Z/2  \\
 -203 & 203 & \Z/4 & \Z^{8}  & \Z/2 \oplus \Z/3 \\
 -211 & 211 & \Z/3 & \Z^{7}  & 0  \\
 -212 & 53 & \Z/6 & \Z^{8}  & (\Z/2)^{2} \oplus \Z/3 \\
 -215 & 215 & \Z/14 & \Z^{4}  & \Z/2 \oplus \Z/3 \\
\hline 
  \end{array}$
 $\begin{array}
  {|c|c|c|c|c|}
\hline
\Delta & m & { \aTop{\rm class}{\rm group}}
 & \Homol_1^{\rm \overline{cusp}}
 & \aTop{\rm Farrell}{\rm supplement} \\
\hline &&&&\\
 -219 & 219 & \Z/4 & \Z^{9}  & \Z/2 \oplus \Z/3 \\
 -223 & 223 & \Z/7 & \Z^{8}  & (\Z/2)^{3}  \\
 -227 & 227 & \Z/5 & \Z^{7}  & \Z/3 \\
 -228 & 57 & \Z/2 \times \Z/2 & \Z^{12}  & (\Z/2)^{3} \oplus (\Z/3)^{2} \\
 -231 & 231 & \Z/6 \times \Z/2 & \Z^{9}  & (\Z/2)^{2} \oplus (\Z/3)^{2} \\
 -232 & 58 & \Z/2 & \Z^{10}  & (\Z/2)^{2} \oplus \Z/3 \\
 -235 & 235 & \Z/2 & \Z^{11}  & (\Z/2)^{3} \oplus \Z/3 \\
 -239 & 239 & \Z/15 & \Z^{3}  & \Z/2 \oplus \Z/3 \\
 -244 & 61 & \Z/6 & \Z^{9}  & (\Z/2)^{2}  \\
 -247 & 247 & \Z/6 & \Z^{8}  & \Z/2  \\
 -248 & 62 & \Z/8 & \Z^{8}  & (\Z/2)^{2} \oplus \Z/3 \\
 -251 & 251 & \Z/7 & \Z^{7}  & \Z/3 \\
 -255 & 255 & \Z/6 \times \Z/2 & \Z^{11}  & (\Z/2)^{2} \oplus (\Z/3)^{3} \\
 -259 & 259 & \Z/4 & \Z^{10}  & \Z/2 \oplus \Z/3 \\
 -260 & 65 & \Z/4 \times \Z/2 & \Z^{12}  & (\Z/2)^{5} \oplus (\Z/3)^{2} \\
 -263 & 263 & \Z/13 & \Z^{5}  & \Z/2 \oplus \Z/3 \\
 -264 & 66 & \Z/4 \times \Z/2 & \Z^{12}  & (\Z/2)^{2} \oplus (\Z/3)^{3} \\
 -267 & 267 & \Z/2 & \Z^{13}  & \Z/3 \\
 -271 & 271 & \Z/11 & \Z^{6}  & \Z/2  \\
 -276 & 69 & \Z/4 \times \Z/2 & \Z^{15}  & (\Z/2)^{3} \oplus (\Z/3)^{2} \\
 -280 & 70 & \Z/2 \times \Z/2 & \Z^{15}  & (\Z/2)^{3} \oplus (\Z/3)^{2} \\
 -283 & 283 & \Z/3 & \Z^{10}  & 0  \\
 -287 & 287 & \Z/14 & \Z^{7}  & (\Z/2)^{2} \oplus \Z/3 \\
 -291 & 291 & \Z/4 & \Z^{13}  & \Z/2 \oplus \Z/3 \\
 -292 & 73 & \Z/4 & \Z^{12}  & (\Z/2)^{2} \oplus \Z/3 \\
 -295 & 295 & \Z/8 & \Z^{11}  & \Z/2 \oplus \Z/3 \\
 -296 & 74 & \Z/10 & \Z^{9}  \oplus (\Z/2)^2 & (\Z/2)^{2} \oplus (\Z/3)^{2} \\
 -299 & 299 & \Z/8 & \Z^{10}  & \Z/2 \oplus (\Z/3)^{4} \\
 -303 & 303 & \Z/10 & \Z^{12}  & \Z/2 \oplus \Z/3 \\
 -307 & 307 & \Z/3 & \Z^{11}  & 0  \\
 -308 & 77 & \Z/4 \times \Z/2 & \Z^{15}  & (\Z/2)^{3} \oplus (\Z/3)^{2} \\
 -311 & 311 & \Z/19 & \Z^{4}  & \Z/2 \oplus \Z/3 \\
 -312 & 78 & \Z/2 \times \Z/2 & \Z^{18}  & (\Z/2)^{3} \oplus (\Z/3)^{2} \\
 -319 & 319 & \Z/10 & \Z^{10}  & \Z/2 \oplus \Z/3 \\
 -323 & 323 & \Z/4 & \Z^{12}  & \Z/2 \oplus \Z/3 \\
 -327 & 327 & \Z/12 & \Z^{12}  & \Z/2 \oplus \Z/3 \\
 -328 & 82 & \Z/4 & \Z^{13}  & (\Z/2)^{3} \oplus \Z/3 \\
 -331 & 331 & \Z/3 & \Z^{12}  & \Z/3 \\
 -335 & 335 & \Z/18 & \Z^{8}  & \Z/2 \oplus \Z/3 \\
 -339 & 339 & \Z/6 & \Z^{15}  & \Z/3 \\
 -340 & 85 & \Z/2 \times \Z/2 & \Z^{19}  & (\Z/2)^{4} \oplus (\Z/3)^{2} \\
 -344 & 86 & \Z/10 & \Z^{11}  \oplus (\Z/2)^2 & \Z/2 \oplus \Z/3 \\
 -347 & 347 & \Z/5 & \Z^{12}  & \Z/3 \\
 -355 & 355 & \Z/4 & \Z^{16}  & \Z/2 \oplus \Z/3 \\
 -356 & 89 & \Z/12 & \Z^{12}  & (\Z/2)^{2} \oplus \Z/3 \\
 -359 & 359 & \Z/19 & \Z^{6}  \oplus (\Z/2)^2 & (\Z/2)^{3} \oplus \Z/3 \\
 -367 & 367 & \Z/9 & \Z^{11}  \oplus (\Z/3)^2 & \Z/2 \oplus \Z/3 \\
 -371 & 371 & \Z/8 & \Z^{14}  & \Z/2 \oplus \Z/3 \\
 -372 & 93 & \Z/2 \times \Z/2 & \Z^{23}  & (\Z/2)^{3} \oplus (\Z/3)^{2} \\
 -376 & 94 & \Z/8 & \Z^{14}  & (\Z/2)^{2} \oplus \Z/3 \\
 -379 & 379 & \Z/3 & \Z^{14}  & 0  \\
 -383 & 383 & \Z/17 & \Z^{8}  & \Z/2 \oplus \Z/3 \\
 -388 & 97 & \Z/4 & \Z^{17}  & (\Z/2)^{2} \oplus \Z/3 \\
 -391 & 391 & \Z/14 & \Z^{11}  & (\Z/2)^{2} \oplus \Z/3 \\
 -395 & 395 & \Z/8 & \Z^{16}  \oplus (\Z/2)^2 & \Z/2 \oplus \Z/3 \\
 -399 & 399 & \Z/8 \times \Z/2 & \Z^{17}  & (\Z/2)^{4} \oplus (\Z/3)^{2} \\
 -403 & 403 & \Z/2 & \Z^{17}  & \Z/2  \\
 -404 & 101 & \Z/14 & \Z^{14}  & (\Z/2)^{4} \oplus \Z/3 \\
 -407 & 407 & \Z/16 & \Z^{13}  & \Z/2 \oplus (\Z/3)^{2} \\
 -408 & 102 & \Z/2 \times \Z/2 & \Z^{23}  & (\Z/2)^{2} \oplus (\Z/3)^{6} \\
 -411 & 411 & \Z/6 & \Z^{19}  & \Z/3 \\
 -415 & 415 & \Z/10 & \Z^{18}  & \Z/2 \oplus \Z/3 \\
 -419 & 419 & \Z/9 & \Z^{13}  & (\Z/3)^{3} \\
\hline 
  \end{array}$\normalsize
}
\caption{The cusp-complementary homology $\Homol_1^{\rm \overline{cusp}}\left(_\Gamma\backslash \Hy; \thinspace \Z \right)$
 for the absolute values of the discriminant $\Delta$ fulfilling $|\Delta| \leq 415$.}
\label{small absolute values}
\end{figure}
\begin{figure}\scriptsize
 $$\begin{array}
  {|c|c|c|c|c|}
\hline 
{\rm Discriminant} & m & { {\rm class}\medspace{\rm group}}
 &\Homol_1^{\rm \overline{cusp}}(_\Gamma\backslash \Hy; \thinspace \Z)
 & \aTop{\rm Farrell}{\rm supplement} \\ 
\hline&&&&\\
 -420 & 105 & \Z/2 \times \Z/2 \times \Z/2 & \Z^{33}  & (\Z/2)^{8} \oplus (\Z/3)^{4} \\
 -424 & 106 & \Z/6 & \Z^{17}  \oplus (\Z/2)^2 & (\Z/2)^{2} \oplus \Z/3 \\
 -427 & 427 & \Z/2 & \Z^{19}  & (\Z/2)^{3}  \\
 -431 & 431 & \Z/21 & \Z^{8}  \oplus \Z/2 & \Z/2 \oplus \Z/3 \\
 -435 & 435 & \Z/2 \times \Z/2 & \Z^{27}  & (\Z/2)^{2} \oplus (\Z/3)^{6} \\
 -436 & 109 & \Z/6 & \Z^{19}  \oplus (\Z/2)^2 & (\Z/2)^{2}  \\
 -439 & 439 & \Z/15 & \Z^{11}  & (\Z/2)^{5}  \\
 -440 & 110 & \Z/6 \times \Z/2 & \Z^{20}  & (\Z/2)^{3} \oplus (\Z/3)^{2} \\
 -443 & 443 & \Z/5 & \Z^{16}  & \Z/2 \oplus \Z/3 \\
 -447 & 447 & \Z/14 & \Z^{18}  & \Z/2 \oplus \Z/3 \\
 -451 & 451 & \Z/6 & \Z^{17}  & \Z/3 \\
 -452 & 113 & \Z/8 & \Z^{19}  \oplus (\Z/2)^2 & (\Z/2)^{2} \oplus \Z/3 \\
 -455 & 455 & \Z/10 \times \Z/2 & \Z^{19}  \oplus (\Z/2)^2 & (\Z/2)^{2} \oplus (\Z/3)^{2} \\
 -456 & 114 & \Z/4 \times \Z/2 & \Z^{24}  \oplus (\Z/2)^2 & (\Z/2)^{2} \oplus (\Z/3)^{4} \\
 -463 & 463 & \Z/7 & \Z^{16}  & \Z/2  \\
 -467 & 467 & \Z/7 & \Z^{16}  & \Z/3 \\
 -471 & 471 & \Z/16 & \Z^{18}  & \Z/2 \oplus \Z/3 \\
 -472 & 118 & \Z/6 & \Z^{19}  \oplus (\Z/2)^2 & \Z/2 \oplus \Z/3 \\
 -479 & 479 & \Z/25 & \Z^{8}  \oplus (\Z/3)^2 & \Z/2 \oplus \Z/3 \\
 -483 & 483 & \Z/2 \times \Z/2 & \Z^{29}  & (\Z/2)^{2} \oplus (\Z/3)^{2} \\
 -487 & 487 & \Z/7 & \Z^{17}  \oplus (\Z/13)^2 & \Z/2  \\
 -488 & 122 & \Z/10 & \Z^{18}  \oplus (\Z/2)^2 & (\Z/2)^{2} \oplus (\Z/3)^{2} \\
 -491 & 491 & \Z/9 & \Z^{16}  & \Z/3 \\
 -499 & 499 & \Z/3 & \Z^{19}  \oplus (\Z/3)^2 & (\Z/2)^{2}  \\
 -520 & 130 & \Z/2 \times \Z/2 & \Z^{28}  \oplus (\Z/2)^2 & (\Z/2)^{4} \oplus (\Z/3)^{2} \\
 -523 & 523 & \Z/5 & \Z^{19}  & 0  \\
 -532 & 133 & \Z/2 \times \Z/2 & \Z^{29}  & (\Z/2)^{3} \oplus (\Z/3)^{2} \\
 -547 & 547 & \Z/3 & \Z^{21}  \oplus (\Z/2)^2 & (\Z/3)^{2} \\
 -555 & 555 & \Z/2 \times \Z/2 & \Z^{35}  & (\Z/2)^{2} \oplus (\Z/3)^{2} \\
 -568 & 142 & \Z/4 & \Z^{25}  \oplus (\Z/2)^2 & (\Z/2)^{6} \oplus \Z/3 \\
 -571 & 571 & \Z/5 & \Z^{23}  & 0  \\
 -595 & 595 & \Z/2 \times \Z/2 & \Z^{33}  & (\Z/2)^{2} \oplus (\Z/3)^{4} \\
 -619 & 619 & \Z/5 & \Z^{23}  \oplus (\Z/3)^2 & 0  \\
 -627 & 627 & \Z/2 \times \Z/2 & \Z^{35}  & (\Z/3)^{2} \\
 -643 & 643 & \Z/3 & \Z^{27}  & \Z/3 \\
 -667 & 667 & \Z/4 & \Z^{28}  & \Z/2 \oplus \Z/3 \\
 -683 & 683 & \Z/5 & \Z^{26}  & \Z/3 \\
 -691 & 691 & \Z/5 & \Z^{26}  \oplus (\Z/7)^2 & 0  \\
 -696 & 174 & \Z/6 \times \Z/2 & \Z^{38}  & (\Z/2)^{3} \oplus (\Z/3)^{3} \\
 -715 & 715 & \Z/2 \times \Z/2 & \Z^{39}  & (\Z/2)^{2} \oplus (\Z/3)^{2} \\
 -723 & 723 & \Z/4 & \Z^{37}  \oplus (\Z/2)^2 & \Z/2 \oplus \Z/3 \\
 -739 & 739 & \Z/5 & \Z^{28}  & 0  \\
 -760 & 190 & \Z/2 \times \Z/2 & \Z^{42}  \oplus (\Z/2)^2 & (\Z/2)^{3} \oplus (\Z/3)^{2} \\
 -763 & 763 & \Z/4 & \Z^{34}  & \Z/2 \oplus \Z/3 \\
 -787 & 787 & \Z/5 & \Z^{30}  & 0  \\
 -795 & 795 & \Z/2 \times \Z/2 & \Z^{51}  & (\Z/2)^{2} \oplus (\Z/3)^{3} \\
 -883 & 883 & \Z/3 & \Z^{35}  & 0\\
 -907 & 907 & \Z/3 & \Z^{36}  \oplus (\Z/13)^2 & 0  \\
 -947 & 947 & \Z/5 & \Z^{37}  \oplus (\Z/89)^2 & \Z/3 \\
 -955 & 955 & \Z/4 & \Z^{46}  \oplus (\Z/2)^4 \oplus (\Z/3)^2 & \Z/2 \oplus \Z/3 \\
 -1003 & 1003 & \Z/4 & \Z^{44}  \oplus (\Z/3)^2 & \Z/2 \oplus \Z/3 \\
 -1027 & 1027 & \Z/4 & \Z^{44}  \oplus (\Z/2)^2 & \Z/2 \oplus (\Z/3)^{3} \\
 -1051 & 1051 & \Z/5 & \Z^{43}  \oplus (\Z/13)^2 & 0  \\
 -1123 & 1123 & \Z/5 & \Z^{44}  \oplus (\Z/7)^2 & 0  \\
 -1227 & 1227 & \Z/4 & \Z^{65}  \oplus (\Z/2^2)^2 & \Z/2 \oplus \Z/3 \\
 -1243 & 1243 & \Z/4 & \Z^{54}  \oplus (\Z/3)^4 & \Z/2 \oplus \Z/3 \\
 -1387 & 1387 & \Z/4 & \Z^{58}  \oplus (\Z/167)^2 & \Z/2 \oplus (\Z/3)^{3} \\
 -1411 & 1411 & \Z/4 & \Z^{60}  \oplus (\Z/2^4)^2 \oplus (\Z/43)^2 & \Z/2 \oplus \Z/3 \\
 -1507 & 1507 & \Z/4 & \Z^{66}  \oplus (\Z/3)^2 \oplus (\Z/5)^4 & \Z/2 \oplus \Z/3 \\
 -1555 & 1555 & \Z/4 & \Z^{76}  \oplus (\Z/2^2)^8 \oplus (\Z/11)^2 & \Z/2 \oplus \Z/3 \\
 -1723 & 1723 & \Z/5 & \Z^{69}  \oplus (\Z/7)^2 \oplus (\Z/23)^2 \oplus (\Z/883)^2 & 0  \\
 -1747 & 1747 & \Z/5 & \Z^{70}  \oplus (\Z/80737)^2 & (\Z/3)^{2} \\
 -1867 & 1867 & \Z/5 & \Z^{75}  \oplus (\Z/2)^4 \oplus (\Z/7^2)^2 \oplus (\Z/137)^2  & 0  \\
\hline 
  \end{array}$$\normalsize
\caption{\hfill $\Homol_1^{\rm \overline{cusp}}(_\Gamma\backslash \Hy; \thinspace \Z)$,
with its torsion decomposed into prime power factors, for some greater absolute values of the discriminant.}
\label{greater absolute values}
\end{figure}

\section{The Bianchi fundamental polyhedron} \label{The Bianchi fundamental polyhedron}

Let $m$ be a squarefree positive integer and consider the imaginary quadratic number field $\rationals(\sqrt{-m})$
 with ring of integers $\mathcal{O}_{-m}$, which we also just denote by $\mathcal{O}$. 
Consider the familiar action by fractional linear transformations
 (we give an explicit formula for it in lemma~\ref{operationFormula})
 of the group \mbox{$\Gamma := \mathrm{SL_2}(\mathcal{O}) \subset \mathrm{GL}_2(\C)$} on hyperbolic three-space,
 for which we will use the upper-half space model $\Hy$. 
As a set, $$ \Hy = \{ (z,\zeta) \in \C \times \R \medspace | \medspace \zeta > 0 \}. $$

The Bianchi--Humbert theory \cite{Bianchi}, \cite{Humbert} gives a fundamental domain for this action.
We will start by giving a geometric description of it, and the arguments why it is a fundamental domain. \\
\begin{definition}
A pair of elements $(\mu, \lambda) \in \mathcal{O}^2$ is called \emph{unimodular} if the ideal sum $\mu \mathcal{O} +\lambda \mathcal{O}$ equals $\mathcal{O}$.
\end{definition}
The boundary of $\Hy$ is the Riemann sphere $\partial \Hy = \C \cup \{ \infty \}$ (as a set), which contains the complex plane $\C$.
The totally geodesic surfaces in $\Hy$ are the Euclidean vertical planes (we define \emph{vertical} as orthogonal to the complex plane) and the Euclidean hemispheres centred on the complex plane.
\begin{notation} \label{hemispheres}
Given a unimodular pair $(\mu$, $\lambda) \in \mathcal{O}^2$ with $\mu \neq 0$, let $S_{\mu,\lambda} \subset \Hy$ denote the hemisphere given by the equation $|\mu z -\lambda|^2 +|\mu|^2\zeta^2 = 1$.

This hemisphere has centre $\lambda/\mu$ on the complex plane $\C$, and radius $1/|\mu|$.
\label{B}
Let 
\\ $
B:=\bigl\{(z,\zeta) \in \Hy$:
The inequality $|\mu z -\lambda|^2 +|\mu|^2\zeta^2 \geq 1$ 
\begin{flushright} is fulfilled for all unimodular pairs $(\mu$, $\lambda) \in \mathcal{O}^2$ with $\mu \neq 0$ $\bigr\}$.
\end{flushright}
Then $B$ is the set of points in $\Hy$ which lie above or on all hemispheres $S_{\mu,\lambda}$.
\end{notation}

\begin{lemma}[\cite{Swan}] \label{GammaBequalsH}
The set $B$ contains representatives for all the orbits of points under the action of $\mathrm{SL_2}(\mathcal{O})$ on $\Hy$.
\end{lemma}

The action  extends continuously to the boundary \mbox{$\partial \Hy$}, which is a Riemann sphere. \\
In $\Gamma := \mathrm{SL_2}(\mathcal{O}_{-m})$, consider the stabiliser subgroup $\Gamma_\infty$ of the point $\infty \in \partial \Hy$. 
In the cases $m =1$ and $m=3$, the latter group contains some rotation matrices like \scriptsize $\begin{pmatrix} 0 & \sqrt{-1} \\ \sqrt{-1} & 0 \end{pmatrix}$\normalsize, 
which we want to exclude. These two cases have been treated in \cite{Mendoza}, \cite{SchwermerVogtmann} and others,
 and we assume $m \neq 1$, $m \neq 3$ throughout the remainder of this article. Then,
$$ \Gamma_\infty = \left\{ \pm \begin{pmatrix}1 & \lambda \\ 0 & 1 \end{pmatrix} \thinspace | \medspace \lambda \in \mathcal{O} \right\},$$
which performs translations by the elements of $\ringO$ with respect to the Euclidean geometry of the upper-half space $\Hy$.\\

\begin{notation} \label{fundamentalRectangle}
A fundamental domain for $\Gamma_\infty$ in the complex plane (as a subset of $\partial \Hy$) is given by the rectangle
$$ D_0 := \begin{cases} \{ x +y\sqrt{-m} \in  \C  \medspace | \medspace 0 \leq x \leq 1, \medspace 0 \leq y \leq 1 \}, &
m \equiv 1 \medspace \mathrm{ or } \medspace 2 \mod 4, \\
\{ x +y\sqrt{-m} \in  \C \medspace | \medspace \frac{-1}{2} \leq x \leq \frac{1}{2}, \medspace 0 \leq y \leq \frac{1}{2} \}, &
m \equiv 3 \mod 4.
\end{cases}$$ 
And a fundamental domain for $\Gamma_\infty$ in $\Hy$ is given by
$$ D_\infty := \{ (z, \zeta) \in \Hy \medspace | \medspace z \in D_0 \}.$$
\end{notation}

\begin{definition}
We define the \emph{Bianchi fundamental polyhedron} as 
$$D := D_\infty \cap B.$$
\end{definition}
It is a polyhedron in hyperbolic space up to the missing vertex $\infty$, 
and up to missing vertices at the singular points if $\ringO$ is not a principal ideal domain (see subsection \ref{singular}).
As Lemma \ref{GammaBequalsH} states $\Gamma \cdot B = \Hy$, and as \mbox{$\Gamma_\infty \cdot D_\infty = \Hy$} yields $\Gamma_\infty \cdot D = B$, we have $\Gamma \cdot D  = \Hy$.
We observe the following notion of strictness of the fundamental domain: the interior of the Bianchi fundamental polyhedron contains no two points which are identified by~$\Gamma$.
\\
Swan proves the following theorem, which implies that the boundary of the Bianchi fundamental polyhedron consists of finitely many cells. 
\begin{theorem}[\cite{Swan}] \label{finiteNumber}
There is only a finite number of unimodular pairs $(\lambda, \mu)$ 
 such that the intersection of $S_{\mu,\lambda}$ with the Bianchi fundamental polyhedron is non-empty.
\end{theorem}
Swan further proves a corollary, from which it can be deduced that the action of $\Gamma$ on $\Hy$ is properly discontinuous.

\section{The Fl\"oge cellular complex} \label{Floege cellular complex}

In order to obtain a cell complex with compact quotient space, we proceed in the following way due to Fl\"oge \cite{Floege}.
The boundary of $\Hy$ is the Riemann sphere~$\partial \Hy$, which, as a topological space,
 is made up of the complex plane $\C$ compactified with the cusp~$\infty$.
The totally geodesic surfaces in $\Hy$ are the Euclidean vertical planes (we define \emph{vertical} as orthogonal to the complex plane) and the Euclidean hemispheres centred on the complex plane.
The action of the Bianchi groups extends continuously to the \mbox{boundary \mbox{$\partial \Hy$}.} 
Consider the cellular structure on $\Hy$ induced by the $\Gamma$-images of the Bianchi fundamental polyhedron.
The cellular closure of this cell complex in $\Hy \cup \partial \Hy $ consists of \mbox{$ \Hy$} and 
\mbox{$\left(\rationals (\sqrt{-m}) \cup \{\infty\}\right) \subset \left(\C \cup \{\infty\}\right) \cong \partial \Hy$.} The $\mathrm{SL_2}(\ringO_{-m})$--orbit of a cusp $\frac{\lambda}{\mu}$ in $\left(\rationals (\sqrt{-m}) \cup \{\infty\}\right)$ corresponds to the ideal class $[(\lambda, \mu)]$ of $\ringO_{-m}$. It is well-known that this does not depend on the choice of the representative $\frac{\lambda}{\mu}$.
We extend our cell complex to a cell complex $\widetilde{X}$ by joining to it, in the case that $\ringO_{-m}$ is not a principal ideal domain, the $\mathrm{SL_2}(\ringO_{-m})$--orbits  of the cusps $\frac{\lambda}{\mu}$ for which the ideal $(\lambda, \mu)$ is not principal. 
At these cusps, we equip $\widetilde{X}$ with the ``horoball topology'' described in~\cite{Floege}. This simply means that the set of cusps, which is discrete in \mbox{$\partial \Hy$}, is located at the hyperbolic extremities of $\widetilde{X}$ : No neighbourhood of a cusp, except the whole $\widetilde{X}$, contains any other cusp.

We retract $\widetilde{X}$ in the following, $\mathrm{SL_2}(\ringO_{-m})$--equivariant, way.
On the Bianchi fundamental polyhedron, the retraction is given by the vertical projection (away from the cusp $\infty$) onto its facets which are closed in $\Hy \cup \partial \Hy$. The latter are the facets which do not touch the cusp $\infty$, and are the bottom facets with respect to our vertical direction. The retraction is continued on $\Hy$ by the group action. It is proven in \cite{FloegePhD} that this retraction is continuous.
We call the retract of $\widetilde{X}$ the \emph{Fl\"oge cellular complex} and denote it by $X$.
So in the principal ideal domain cases, $X$ is a retract of the original cellular structure on~$\Hy$, obtained by contracting the Bianchi fundamental polyhedron onto its cells which do not touch the boundary of $\Hy$.  
In \cite{RahmFuchs}, it is checked that the Fl\"oge cellular complex is contractible.


\section{Connecting Fl\"oge cell complex and Borel--Serre compactification} \label{connection}

Let $\Gamma$ be a Bianchi group with $\ringO$ admitting as only units $\{\pm 1\}$,
 i.e. we suppose $\ringO$ not to be the Gaussian or Eisenstein integers.
 In the latter two cases, the problem of the singular cusps treated here does not occur anyway.
Let $\T_i$ be the torus attached at the cusp $i$ of~$\Gamma$,
and let $x_i$ and $y_i$ denote the cycles generating $\Homol_1(\T_i)$.
Let $P$ be the Bianchi fundamental polyhedron of~$\Gamma$.
Write ``hyp. cells'' for cells in the interior of hyperbolic space.
Consider the short exact sequence of chain complexes obtained from collapsing the singular tori,
\begin{flushleft}
 \begin{tikzpicture}[descr/.style={fill=white,inner sep=1.5pt}]
        \matrix (m) [
            matrix of math nodes,
            row sep=2em,
            column sep=1.4em,
            text height=2.5ex, text depth=0.95ex
        ]
        {0 & 0 & \langle P \rangle  & \langle P \rangle & 0 \\
           0 & \bigoplus\limits_s^{\rm singular} \langle \T_s \rangle \quad \medspace & \bigoplus\limits_c^{\rm any \medspace cusp} \langle \T_c \rangle \oplus \langle {\rm hyp. \medspace }2{\rm -cells} \rangle & \langle \T_\infty \rangle \oplus \langle {\rm hyp. \medspace }2{\rm -cells} \rangle & 0 \\ 
	    0 & \bigoplus\limits_s^{\rm singular} \langle x_s, y_s \rangle & \bigoplus\limits_c^{\rm any \medspace cusp} \langle x_c, y_c \rangle \oplus \langle {\rm hyp. \medspace }{\rm edges} \rangle & \langle x_\infty, y_\infty \rangle \oplus \langle {\rm hyp. \medspace }{\rm edges} \rangle & 0 \\   
	     0 & 0 & \bigoplus\limits_c^{\rm any \medspace cusp} \langle c \rangle \oplus \langle {\rm hyp. \medspace }{\rm vertices} \rangle & \bigoplus\limits_c^{\rm any \medspace cusp} \langle c \rangle \oplus \langle {\rm hyp. \medspace }{\rm vertices} \rangle & 0. \\  
        };

        \path[overlay,->, font=\scriptsize,>=latex]
        (m-1-1) edge (m-1-2) 
        (m-1-2) edge (m-2-2)
        (m-1-2) edge (m-1-3) 
        (m-1-3) edge node[descr,xshift=1.4ex] {$\partial_3$}(m-2-3)
        (m-1-3) edge (m-1-4) 
        (m-1-4) edge node[descr,xshift=1.4ex] {$\widetilde{\partial_3}$} (m-2-4)
        (m-1-4) edge (m-1-5)
        (m-2-1) edge (m-2-2)
	(m-2-2) edge node[descr,yshift=1.2ex] {$\beta$} (m-2-3) 
        (m-2-3) edge (m-2-4)
        (m-2-4) edge (m-2-5)
	(m-2-2) edge node[descr,xshift=1.2ex] {$0$} (m-3-2)
        (m-2-3) edge node[descr,xshift=3.0ex] {$0 \oplus \partial_2$} (m-3-3)
        (m-2-4) edge node[descr,xshift=3.0ex] {$0 \oplus \widetilde{\partial_2}$} (m-3-4) 
        (m-3-1) edge (m-3-2)
	(m-3-2) edge node[descr,yshift=1.4ex] {$\beta$} (m-3-3) 
        (m-3-3) edge (m-3-4)
        (m-3-4) edge (m-3-5)
	(m-3-2) edge (m-4-2)
        (m-3-3) edge node[descr,xshift=3.0ex] {$0 \oplus \partial_1$} (m-4-3)
        (m-3-4) edge node[descr,xshift=3.0ex] {$0 \oplus \partial_1$} (m-4-4) 
        (m-4-1) edge (m-4-2)
	(m-4-2) edge (m-4-3) 
        (m-4-3) edge (m-4-4)
        (m-4-4) edge (m-4-5)
  ;
\end{tikzpicture}
\end{flushleft}
Poincar\'e's theorem on fundamental polyhedra tells us that $\partial_3(P) = \left\langle \bigcup\limits_c^{\rm any \medspace cusp} \T_c \right\rangle$, 
and hence $\widetilde{\partial_3}(P) = \langle \T_\infty \rangle$.
From \cite{On a question of Serre}, we see that for every cusp $c$, there is a chain of hyperbolic $2$-cells
that we denote by $ch(x_c)$ and which is mapped to the cycle $x_c$ by $\partial_2$.
And furthermore, $y_c$ is in the cokernel of $\partial_2$ 
(of course, this holds up to the appropriate permutation of the labels $x_c$ and $y_c$).
This implies that  $\widetilde{\partial_2}(ch(x_\infty)) = x_\infty$ and $y_\infty$ is in the cokernel of $\widetilde{\partial_2}$.
As the quotient space is path-wise connected, the cokernel of $\partial_1$ is isomorphic to $\Z$.
The above information tells us that the long exact sequence induced on integral homology by the map $\beta$ concentrates in

\medskip

\begin{tikzpicture}[descr/.style={fill=white,inner sep=1.5pt}]
        \matrix (m) [
            matrix of math nodes,
            row sep=1em,
            column sep=2.5em,
            text height=1.7ex, text depth=0.25ex
        ]
        {  0 & \bigoplus\limits_s^{\rm singular} \langle \T_s \rangle &  \bigl(\bigoplus\limits_c^{\rm any \medspace cusp} \langle \T_c \rangle\bigr) /_{\langle \cup_c \T_c \rangle} \oplus \Homol_2^{\rm \overline{cusp}} & \Homol_2^{\rm \overline{cusp}}  \bigoplus_s \langle ch(x_s) \rangle \\
            & \bigoplus_s \langle x_s, y_s \rangle & \bigoplus_c \langle y_c \rangle   \oplus \Homol_1^{\rm \overline{cusp}} & \Homol_1^{\rm \overline{cusp}} \oplus \langle y_\infty \rangle \\
            & 0  & \Z & \Z & 0, \\
        };

        \path[overlay,->, font=\scriptsize,>=latex]
        (m-1-1) edge (m-1-2) 
        (m-1-2) edge node[descr,yshift=1.2ex] {$\beta_2$} (m-1-3) 
        (m-1-3) edge (m-1-4) 
        (m-1-4) edge[out=355,in=175]  (m-2-2)
        (m-2-2) edge node[descr,yshift=1.15ex] {$\beta_1$} (m-2-3)
        (m-2-3) edge (m-2-4)
        (m-2-4) edge[out=355,in=175]  (m-3-2)
        (m-3-2) edge (m-3-3)
        (m-3-3) edge (m-3-4)
        (m-3-4) edge (m-3-5);
\end{tikzpicture}

where the maps without labels are the obvious restriction maps making the sequence exact;
 and where $\Homol_1^{\rm \overline{cusp}}$ and $\Homol_2^{\rm \overline{cusp}}$ are generated by cycles from the interior of~$_\Gamma\backslash \Hy$.

Note that $\Homol_2^{\rm \overline{cusp}}  \bigoplus_s \langle ch(x_s) \rangle$ is non-naturally isomorphic to  $\left(\bigoplus\limits_c^{\rm any \medspace cusp} \langle \T_c \rangle\right) /_{\langle \cup_c \T_c \rangle} \oplus \Homol_2^{\rm \overline{cusp}}$,
namely collapsing a torus $\T_s$ moves its $2$-cycle into a bubble $ch(x_s)$ emerging adjacent to the singular cusp $s$ in the Fl\"oge complex.

\section{ The equivariant spectral sequence to group homology} \label{SpecSeq}
Let $\Gamma := \mathrm{PSL_2}(\ringO_{-m})$, and let $X$ be the Fl\"oge cellular complex of section~\ref{Floege cellular complex},
 the cell structure of which we subdivide until the cells are fixed pointwise by their stabilisers.
 We describe now how to assemble the homology of the Borel--Serre compactified quotient space (issue of the previous section)
and the Farrell cohomology of $\Gamma$, for which general formulae have been given in~\cite{accessing Farrell cohomology}
(based on~\cite{homological torsion}),
to the full group homology of $\Gamma$ with trivial $\Z$--coefficients.
 We proceed following \cite[VII]{Brown} and \cite{SchwermerVogtmann}.
Let us consider the homology $\Homol_*(\Gamma; C_\bullet(X))$ of $\Gamma$ with coefficients in the cellular chain complex $C_\bullet(X)$ associated to $X$; and call it the $\Gamma$-equivariant homology of $X$.
As $X$ is contractible, the map $X \to pt.$ to the point $pt.$ induces an isomorphism
$$\Homol_*(\Gamma; \thinspace C_\bullet(X)) \to \Homol_*(\Gamma; \thinspace C_\bullet(pt.)) \cong \Homol_*(\Gamma;\thinspace \Z).$$
Denote by $X^p$ the set of $p$-cells of $X$,
 and make use of that the stabiliser $\Gamma_\sigma$ in $\Gamma$ of any $p$-cell $\sigma$ of $X$ fixes $\sigma$ pointwise. Then from 
$$ C_p(X) = \bigoplus\limits_{\sigma \in X^p} \Z 
\cong \bigoplus\limits_{\sigma \thinspace\in \thinspace _\Gamma \backslash X^p} {\rm Ind}^\Gamma_{\Gamma_\sigma} \Z,$$
Shapiro's lemma yields
$$\Homol_{q}(\Gamma; \thinspace C_p(X)) 
\cong \bigoplus_{\sigma\thinspace\in\thinspace _\Gamma\backslash X^p}\Homol_q(\Gamma_\sigma; \thinspace \Z);$$
and the equivariant Leray/Serre spectral sequence takes the form 
$$
E^1_{p,q}=\bigoplus_{\sigma\thinspace\in\thinspace _\Gamma\backslash X^p}\Homol_q(\Gamma_\sigma; \thinspace \Z)\implies \Homol_{p+q}(\Gamma; \thinspace C_\bullet(X)),
$$
converging to the $\Gamma$-equivariant homology of $X$, which is, as we have already seen, isomorphic to 
$\Homol_{p+q}(\Gamma; \thinspace \Z)$ with the trivial action on the coefficients $\Z$.

As in degrees above the virtual cohomological dimension, which is $2$ for the Bianchi groups,
 the group homology is isomorphic to the Farrell cohomology, we obtain the isomorphism type from the above mentioned general formulae.

In the lower degrees $q \in \{0,1,2\}$, the following terms remain on the $E^2$-page,  which is concentrated in the columns $p=0,1,2$:
\scriptsize
$$ \xymatrix{
q=2 & \bigoplus\limits_{s\medspace{\rm singular}} \Z \oplus  2\mbox{-}{\rm torsion} \oplus 3\mbox{-}{\rm torsion} &  2\mbox{-}{\rm torsion} \oplus 3\mbox{-}{\rm torsion} & 0\\
q=1 & \bigoplus\limits_{s\medspace{\rm singular}} \Z^2 \oplus {\rm Farrell \medspace supplement} & 2\mbox{-}{\rm torsion} \oplus 3\mbox{-}{\rm torsion} & 0 \\
q=0 &{\mathbb{Z}} & \Homol_1(_\Gamma\backslash{X}; \thinspace \Z) & \ar[ull]_{d^2_{2,0}} \Homol_2(_\Gamma\backslash{X}; \thinspace \Z) 
}$$
\normalsize
where the ``Farrell supplement'' is the cokernel of the map
$$
\bigoplus_{\sigma \thinspace\in \thinspace _\Gamma\backslash X^0} \Homol_1 (\Gamma_\sigma; \Z)
\xleftarrow{\ d^1_{1,1}\ }
\bigoplus_{\sigma \thinspace\in \thinspace _\Gamma\backslash X^1} \Homol_1(\Gamma_\sigma; \Z).
$$
induced by inclusion of finite cell stabilisers.
As the cells are fixed pointwise by their stabilisers, we see that for $q>0$, the $E^1_{p,q}$-terms are concentrated in the two columns $p = 0$ and $p = 1$. 
We compute the bottom row ($q=0$)  of the above spectral sequence as the homology of the quotient space $_\Gamma\backslash{X}$.
Then we infer from section~\ref{connection} that the rational rank of the differential $d^2_{2,0}$
is the number of non-trivial ideal classes of $\ringO_{-m}$. 



Using Serre's decomposition of the homology group $\Homol_1(_\Gamma\backslash \Hy; \thinspace \Z)$
 into the direct sum of the free Abelian group with one generator for each element of the class group of $\ringO_{-m}$
and the group $\Homol_1^{\rm \overline{cusp}}(_\Gamma\backslash \Hy; \thinspace \Z)$,
and using the long exact sequence of section~\ref{connection},
 we see that
 $\Homol_1(_\Gamma\backslash X; \thinspace \Z) 
\cong \Homol_1^{\rm \overline{cusp}}(_\Gamma\backslash \Hy; \thinspace \Z) \oplus \Z$.
This has made it possible ot compute $\Homol_1^{\rm \overline{cusp}}(_\Gamma\backslash \Hy; \thinspace \Z)$
from the quotient space of the Fl\"oge cellular complex in figures~\ref{small absolute values} and~\ref{greater absolute values}.
Furthermore, we have an inclusion of $\Homol_1(_\Gamma\backslash \Hy; \thinspace \Z)$ into the group homology
 $\Homol_1(\Gamma; \thinspace \Z)$;
and the latter group homology is a quotient of the direct sum of $\Homol_1(_\Gamma\backslash \Hy; \thinspace \Z)$
and the Farrell supplement.
\section{Swan's concept to determine the Bianchi fundamental polyhedron} \label{Swan's concept}

This section recalls Richard G. Swan's work \cite{Swan}, 
which gives a concept --- from the  theoretical viewpoint --- for an algorithm to compute the Bianchi fundamental polyhedron. 
The set $B$ which determines the Bianchi fundamental polyhedron has been defined using infinitely many hemispheres. But we will see that only a finite number of them are significant for this purpose and need to be computed. We will state a criterion for what is an appropriate choice that gives us precisely the set $B$. This criterion is easy to verify in practice.
Suppose we have made a finite selection of $n$ hemispheres. The index $i$ running from $1$ through $n$, we denote the $i$-th hemisphere by $S(\alpha_i)$, where $\alpha_i$ is its centre and given by a fraction $\alpha_i = \frac{\lambda_i}{\mu_i}$ in the number field $\rationals(\sqrt{-m}\thinspace)$.
 Here, we require the ideal $(\lambda_i,\mu_i)$ to be the whole ring of integers $\mathcal{O}$. This requirement is just the one already made for all the hemispheres in the definition of $B$.
Now, we can do an approximation of notation \ref{B}, using, modulo the translation group $\Gamma_\infty$, a finite number of hemispheres.  

\begin{notation} \label{spaceAboveHemispheres}
Let $
B(\alpha_1,\ldots,\alpha_n):=\bigl\{(z,\zeta) \in \Hy$:
The inequality $|\mu z -\lambda|^2 +|\mu|^2\zeta^2 \geq 1$  is fulfilled for all unimodular pairs \mbox{$(\mu$, $\lambda) \in \mathcal{O}^2$} with $\frac{\lambda}{\mu} = \alpha_i +\gamma$, for some $i \in \{1,\ldots,n\}$ and some $\gamma \in \mathcal{O}$ $\bigr\}$.
Then $B(\alpha_1,\ldots,\alpha_n)$ is the set of all points in $\Hy$ lying above or on all hemispheres $S(\alpha_i +\gamma)$, $i = 1,\ldots,n$; for any $\gamma \in \mathcal{O}$.
\end{notation}
The intersection $B(\alpha_1,\ldots,\alpha_n) \cap D_\infty$  with the fundamental domain $D_\infty$ for the translation group $\Gamma_\infty$, is our candidate to equal the Bianchi fundamental polyhedron.

\subsection{Convergence of the approximation.} \hfill
We will give a method to decide when \\ \mbox{$B(\alpha_1,\ldots,\alpha_n) = B$.}
This gives us an effective way to find $B$ by adding more and more elements to the set $\{ \alpha_1,\ldots,\alpha_n \}$ until we find $B(\alpha_1,\ldots,\alpha_n) = B$. 
\label{delB}
We consider the boundary $\partial B(\alpha_1,\ldots,\alpha_n)$ of $B(\alpha_1,\ldots,\alpha_n)$ in $\Hy \cup \C$. It consists of the points $(z,\zeta) \in \Hy \cup \C$ satisfying all the non-strict inequalities
$|\mu z -\lambda|^2 +|\mu|^2\zeta^2 \geq 1$ that we have used to define $B(\alpha_1,\ldots,\alpha_n)$, and satisfy the additional condition that at least one of these non-strict inequalities is an equality.
We will see below that $\partial B(\alpha_1,\ldots,\alpha_n)$ carries a natural cell structure.
This, together with the following definitions, makes it possible to state the criterion which tells us when we have found all the hemispheres relevant for the Bianchi fundamental polyhedron.

\begin{definition} \label{strictlybelow}
We shall say that the hemisphere $S_{\mu,\lambda}$  is \emph{strictly below} the hemisphere $S_{\beta,\alpha}$  at a point $z \in \C$ if the following inequality is satisfied:
$$ \left| z - \frac{\alpha}{\beta} \right|^2 - \frac{1}{|\beta|^2} 
<  \left| z - \frac{\lambda}{\mu}  \right|^2 - \frac{1}{|\mu|^2}.$$
\end{definition}
 This is, of course, an abuse of language because there may not be any points on $S_{\beta,\alpha}$ or $S_{\mu,\lambda}$
 with coordinate $z$. However, if there is a point $(z,\zeta)$ on $S_{\mu,\lambda}$,
 the right hand side of the inequality is just $-\zeta^2$.
Thus the left hand side is negative and so of the form $-(\zeta')^2$. 
Clearly, $(z,\zeta') \in S_{\beta,\alpha}$ and $\zeta' > \zeta$. 
We will further say that a point $(z,\zeta) \in \Hy \cup \C$ is \emph{strictly below} a hemisphere $S_{\mu,\lambda}$,
 if there is a point $(z,\zeta') \in S_{\mu,\lambda}$ with $\zeta' > \zeta$.

\subsection{Singular points} \label{singular}
We call \emph{cusps} the elements of the number field $K = \rationals(\sqrt{-m}\thinspace)$
 considered as points in the boundary of hyperbolic space,
 via an embedding \mbox{$K  \subset \C \cup \{ \infty \} \cong \partial \Hy$.}
We write $\infty = \frac{1}{0}$, which we also consider as a cusp.
It is well-known that the set of cusps is closed under the action of SL$_2(\ringO)$ on $\partial \Hy$;
and that we have the following bijective correspondence between 
the SL$_2(\ringO)$-orbits of cusps and the ideal classes in $\ringO$. 
A cusp $\frac{\lambda}{\mu}$ is in the SL$_2(\ringO)$-orbit of the cusp $\frac{\lambda'}{\mu'}$,
 if and only if the ideals $(\lambda',\mu')$ and $(\lambda,\mu)$ are in the same ideal class.
It immediately follows that the orbit of the cusp $\infty = \frac{1}{0}$ corresponds to the principal ideals.
Let us call \emph{singular} the cusps $\frac{\lambda}{\mu}$ such that $(\lambda,\mu)$ is not principal.
And let us call \emph{singular points} the singular cusps which lie in $\partial B$.
It follows from the characterisation of the singular points by Bianchi that they are precisely the points in
 $\C \subset \partial \Hy$ which cannot be strictly below any hemisphere.
In the cases where $\ringO$ is a principal ideal domain, $K \cup \{ \infty \}$ consists of only one SL$_2(\ringO)$-orbit,
 so there are no singular points. 
We use the following formulae derived by Swan, to compute representatives modulo the translations by $\Gamma_\infty$,
 of the singular points.  
\begin{lemma}[\cite{Swan}] 
The singular points of $K, \mod \ringO$, are given by \\
$\frac{p(r+\sqrt{-m})}{s}$, where 
\mbox{$p,r,s \in \Z$,} $s>0$, \qquad
$\frac{-s}{2} < r \leq \frac{s}{2}$, \qquad $s^2 \leq r^2 +m$, and
\begin{itemize}
\item if $m \equiv 1$ or $2 \mod 4$, 
\\ $s \neq 1$, $s \medspace | \medspace r^2+m$, the numbers $p$ and $s$ are coprime,
and $p$ is taken $\mod s$;
\item if $m \equiv 3 \mod 4$, 
\\ $s$ is even, $s \neq 2$, $2s \medspace | \medspace r^2 +m$, the numbers $p$ and $\frac{s}{2}$ are coprime;
$p$ is taken$ \mod \frac{s}{2}$.
\end{itemize} 
\end{lemma}
The singular points need not be considered in Swan's termination criterion, because they cannot be strictly below any hemisphere $S_{\mu,\lambda}$.

\subsection{Swan's termination criterion}
We observe that the set of $z \in \C$ over which some hemisphere is strictly below another is $\C$ or an open half-plane.
In the latter case, the boundary of this is a line.

\begin{notation} \label{agreeingLine}
Denote by $L(\frac{\alpha}{\beta}, \frac{\lambda}{\mu})$ the set of $z \in \C$ over which neither $S_{\beta,\alpha}$ is strictly below $S_{\mu,\lambda}$ nor vice versa.
\end{notation}
This line is computed by turning the inequality in definition \ref{strictlybelow} into an equation. Swan calls it the line over which the two hemispheres \emph{agree}, and we will see later that the most important edges of the Bianchi fundamental polyhedron lie on the preimages of such lines. 
We now restrict our attention to a set of hemispheres which is finite modulo the translations in $\Gamma_\infty$.
\label{collection}
Consider a set of hemispheres $S(\alpha_i +\gamma)$, where the index $i$ runs from 1 through $n$, and $\gamma$ runs through $\ringO$. We call this set of hemispheres a \emph{collection}, if every non-singular 
point $z \in \C \subset \partial \Hy$ is strictly below some hemisphere in our set.
Now consider a set $B(\alpha_1,\ldots,\alpha_n)$ which is determined by such a collection of hemispheres.

\begin{theorem}[Swan's termination criterion \cite{Swan}] \label{cancelCriterion} We have 
 \mbox{$B(\alpha_1,\ldots,\alpha_n) = B$} if and only if no vertex of $\partial B(\alpha_1,\ldots,\alpha_n)$ can be strictly below any hemisphere $S_{\mu,\lambda}$.
\end{theorem}
In other words, no vertex $v$ of $\partial B(\alpha_1,\ldots,\alpha_n)$ can lie strictly below any hemisphere $S_{\mu,\lambda}$.
\\
Let us call \emph{height} the coordinate $\zeta$
 of the upper-half space model introduced at the beginning of section~\ref{The Bianchi fundamental polyhedron}.
With this criterion, it suffices to compute the cell structure of $\partial B(\alpha_1,\ldots,\alpha_n)$ to see if our choice of hemispheres gives us the Bianchi fundamental polyhedron.
This has only to be done modulo the translations of $\Gamma_\infty$, which preserve the height and hence the situations of being strictly below.
Thus our computations only need to be carried out on a finite set of hemispheres.

\subsection{Computing the cell structure in the complex plane} \label{cellStructure}

We will in a first step compute the image of the cell structure under the homeomorphism from $\partial B(\alpha_1,\ldots,\alpha_n)$ to $\C$ given by the vertical projection.
For each 2-cell of this structure, there is an associated hemisphere $S_{\mu,\lambda}$.
 The interior of this 2-cell consists of the points $z \in \C$,
 where all other hemispheres in our collection are strictly below  $S_{\mu,\lambda}$.
 Swan shows that this is the interior of a convex polygon.
The edges of these polygons lie on real lines in $\C$ specified in notation~\ref{agreeingLine}.

A vertex is an intersection point $z$ of any two of these lines involving the same hemisphere $S_{\mu,\lambda}$,
 if  all other hemispheres in our collection are strictly below, or agree with, $S_{\mu,\lambda}$ at~$z$.

\subsection{Lifting the cell structure back to hyperbolic space}
Now we can lift the cell structure back to $\partial B(\alpha_1,\ldots,\alpha_n)$, using the projection homeomorphism onto $\C$.  
The preimages of the convex polygons of the cell structure on $\C$, 
are totally geodesic hyperbolic polygons each lying on one of the hemispheres in our collection. 
These are the 2-cells of $\partial B(\alpha_1,\ldots,\alpha_n)$.

The edges of these hyperbolic polygons lie on the intersection arcs of pairs of hemispheres in our collection.
 As two Euclidean 2-spheres intersect, if they do so non-trivially,
 in a circle centred on the straight line which connects the two 2-sphere centres,
 such an intersection arc lies on a semicircle centred in the complex plane.
 The plane which contains this semicircle must be orthogonal to the connecting line, hence a vertical plane in $\Hy$.
 We can alternatively conclude the latter facts observing that an edge which two totally geodesic polygons have in common must be a geodesic segment.
Lifting the vertices becomes now obvious from their definition. This enables us to check Swan's termination criterion.

We will now sketch Swan's proof of this criterion.
Let $P$ be one of the convex polygons of the cell structure on $\C$. 
The preimage of $P$ lies on one hemisphere $S(\alpha_i)$ of our collection. 
Now the condition stated in theorem \ref{cancelCriterion} says that at the vertices of $P$, the hemisphere $S(\alpha_i)$ cannot be strictly below any other hemisphere.
The points where $S(\alpha_i)$  can be strictly below some hemisphere constitute an open half-plane in $\C$, and hence cannot lie in the convex hull of the vertices of $P$, which is $P$.
Theorem \ref{cancelCriterion} now follows because $\C$ is tessellated by these convex polygons.

\newpage

\section{Algorithms realizing Swan's concept} \label{realization}

From now on, we will work on putting Swan's concept into practice. 
We can reduce the set of hemispheres on which we carry out our computations, with the help of the following notion.
\begin{definition} \label{everywherebelow}
A hemisphere $S_{\mu,\lambda}$ is said to be \emph{everywhere below} a hemisphere $S_{\beta,\alpha}$  when: 
$$\vline \frac{\lambda}{\mu} -\frac{\alpha}{\beta}\vline \leq \frac{1}{|\beta|} - \frac{1}{|\mu|.}$$
\end{definition}
Note that this is also the case when $S_{\mu,\lambda} = S_{\beta,\alpha}$. 
Any hemisphere which is everywhere below another one, does not contribute to the Bianchi fundamental polyhedron, in the following sense.

\begin{proposition} \label{everywhereBelowHemispheres}
Let  $S(\alpha_n)$ be a hemisphere everywhere below some other hemisphere  $S(\alpha_i)$, where $i \in \{1,\ldots,n-1\}$. \\
Then  $B(\alpha_1,\ldots,\alpha_n) = B(\alpha_1,\ldots,\alpha_{n-1})$.
\end{proposition}
\begin{proof}
Write $\alpha_n = \frac{\lambda}{\mu}$ and $\alpha_i = \frac{\theta}{\tau}$ with 
$\lambda,\mu,\theta,\tau \in \ringO$.
We take any point $(z,\zeta)$ strictly below $S_{\mu,\lambda}$ and show that it is also strictly below $S_{\tau,\theta}$.
In terms of notation \ref{spaceAboveHemispheres}, this problem looks as follows: we assume that the inequality $|\mu z-\lambda|^2 +|\mu|^2 \zeta^2 < 1$  is satisfied, and show that this implies the inequality $|\tau z-\theta|^2 +|\tau|^2 \zeta^2 < 1$. 
The first inequality can be transformed into \\
 \ ${\vline z-\frac{\lambda}{\mu}\vline \thinspace}^2 +\zeta^2 < \frac{1}{|\mu|^2}$.
Hence, $\sqrt{{\vline z-\frac{\lambda}{\mu}\vline \thinspace}^2 +\zeta^2} < \frac{1}{|\mu|}$. We will insert this into the triangle inequality for the Euclidean distance in $\C \times \R$ applied to the three points  $(z,\zeta)$, $(\frac{\lambda}{\mu},0)$ and $(\frac{\theta}{\tau},0)$, which is 
$$\sqrt{{\vline z-\frac{\theta}{\tau}\vline \thinspace}^2 +\zeta^2} \less \vline\frac{\lambda}{\mu} -\frac{\theta}{\tau} \vline +\sqrt{{\vline z-\frac{\lambda}{\mu}\vline\thinspace}^2 +\zeta^2}.$$
So we obtain
$\sqrt{{\vline z-\frac{\theta}{\tau}\vline\thinspace}^2 +\zeta^2} \less \vline\frac{\lambda}{\mu} -\frac{\theta}{\tau} \vline +\frac{1}{|\mu|}$.
By definition \ref{everywherebelow}, the expression on the right hand side is smaller than or equal to $\frac{1}{|\tau|}$. Therefore, we take the square and obtain
${\vline z-\frac{\theta}{\tau}\vline\thinspace}^2 +\zeta^2 < \frac{1}{|\tau|^2}$,
which is equivalent to the claimed inequality.
\end{proof}

Another notion that will be useful for our algorithm, is the following.

\begin{definition} \label{liftOnHemisphere}
Let $z \in \C$ be a point lying within the vertical projection of $S_{\mu,\lambda}$. Define the \emph{lift on the hemisphere}
 $S_{\mu,\lambda}$ of $z$ as the point on $S_{\mu,\lambda}$ the vertical projection of which is~$z$.
\end{definition}

\begin{notation}
Denote by the \emph{hemisphere list} a list into which we will record a finite number of hemisphere
s $S(\alpha_1)$, $\hdots$, $S(\alpha_n)$. Its purpose is to determine a set $B(\alpha_1,\ldots,\alpha_n)$ in order to approximate, and finally obtain, the Bianchi fundamental polyhedron.
\end{notation}

\newpage
\subsection{The algorithm computing the Bianchi fundamental polyhedron}

We now describe the algorithm that we have realized using Swan's description; it is decomposed into algorithms \ref{Get_Bianchi_Fundamental_Polyhedron} through 
\ref{Get_minimal_proper_vertex_height} below.

\bigskip
\textbf{Initial step.} We begin with the smallest value which the norm of a non-zero element $\mu \in \ringO$ can take, namely 1. Then $\mu$ is a unit in $\ringO$, and for any $\lambda \in \ringO$, the pair  $(\mu,\lambda)$ is unimodular. And we can rewrite the fraction $\frac{\lambda}{\mu}$ such that $\mu = 1$. We obtain the unit hemispheres (of radius 1), centred at the imaginary quadratic integers $\lambda \in \ringO$. We record into the hemisphere list the ones which touch the Bianchi fundamental polyhedron, i.e. the ones the centre of which lies in the fundamental rectangle $D_0$ (of notation \ref{fundamentalRectangle}) for the action of $\Gamma_\infty$ on the complex plane.

\medskip
\textbf{Step A.} Increase $|\mu|$ to the next higher value which the norm takes on elements of $\ringO$. Run through all the finitely many $\mu$ which have this norm. For each of these $\mu$, run through all the finitely many $\lambda$ with $\frac{\lambda}{\mu}$ in the fundamental rectangle $D_0$. Check that $(\mu,\lambda) = \ringO$ and that the hemisphere $S_{\mu,\lambda}$ is not everywhere below a hemisphere $S_{\beta,\alpha}$ in the hemisphere list.
If these two checks are passed, record $(\mu,\lambda)$ into the hemisphere list.

\medskip
We repeat step \textbf{A} until $|\mu|$ has reached an expected value. Then we check if we have found all the hemispheres  which touch the Bianchi fundamental polyhedron, as follows.

\medskip
\textbf{Step B.} 
We compute the lines $L(\frac{\alpha}{\beta},\frac{\lambda}{\mu})$ of definition \ref{agreeingLine}, over which two hemispheres agree, for all pairs $S_{\beta,\alpha}$,  $S_{\mu,\lambda}$ in the hemisphere list which touch one another. \\
Then, for each hemisphere $S_{\beta,\alpha}$, we compute the intersection points of each two lines $L(\frac{\alpha}{\beta},\frac{\lambda}{\mu})$ and $L(\frac{\alpha}{\beta},\frac{\theta}{\tau})$ referring to $\frac{\alpha}{\beta}$. \\
We drop the intersection points at which $S_{\beta,\alpha}$ is strictly below some hemisphere in the list. \\
We erase the hemispheres from our list, for which less than three intersection points remain. We can do this because a hemisphere which touches the \mbox{Bianchi} fundamental polyhedron only in two vertices shares only an edge with it and no 2-cell. \\
Now, the vertices of $B(\alpha_1,\ldots,\alpha_n) \cap D_\infty$ are the lifts of the remaining intersection points.
Thus we can check Swan's termination criterion (theorem \ref{cancelCriterion}), which we do as follows. We pick the lowest value $\zeta > 0$ for which $(z,\zeta) \in \Hy$ is the lift inside Hyperbolic Space of a remaining intersection point $z$. \\ 
If $\zeta \geq \frac{1}{|\mu|}$, then all (infinitely many) remaining hemispheres have radius equal or smaller than $\zeta$, so $(z,\zeta)$ cannot be strictly below them. So Swan's termination criterion is fulfilled, we have found the Bianchi fundamental polyhedron, and can proceed by determining its cell structure. \\
Else, $\zeta$ becomes the new expected value for $\frac{1}{|\mu|}$.
We repeat step \textbf{A} until $|\mu|$ reaches $\frac{1}{\zeta}$ and then proceed again with step \textbf{B}.

\bigskip


\begin{algorithm} 
\caption{Computation of the Bianchi fundamental polyhedron}
\label{Get_Bianchi_Fundamental_Polyhedron}
{
\begin{alginc}

\State \bf Input: \rm A square-free positive integer $m$.
\State \bf Output: \rm The hemisphere list,  containing entries $S(\alpha_1)$,\ldots,$S(\alpha_n)$ such that $B(\alpha_1,\ldots,\alpha_n) = B$. 
\State
\State Let $\ringO$ be the ring of integers in $\rationals(\sqrt{-m})$. 
\State Let $h_\ringO$ be the class number of $\ringO$. Compute $h_\ringO$.
\State Estimate the highest value for $|\mu|$ which will occur in 
notation \ref{spaceAboveHemispheres} by
\State the formula $E$ $ := \begin{cases}
\frac{5m}{2}h_\ringO -2m +\frac{1}{2}, & m \equiv 3 \mod 4, \\
21mh_\ringO -19m,            & \mathrm{else}.
\end{cases}$ 
\State \rm $\mathcal{N} := 1$.
\State Swan's\_cancel\_criterion\_fulfilled $:=$ false.
\State
\While{ Swan's\_cancel\_criterion\_fulfilled $=$ false, }

        \While{ $\mathcal{N} \leq$  $E$ }

                \State Execute algorithm \ref{Record_hemispheres} with argument $\mathcal{N}$.
                \State Increase $\mathcal{N}$ to the next greater value in 
                \State the set $\{\sqrt{n^2 m+j^2} \medspace | \medspace n,j \in \N \}$ 
                of values of the norm on $\ringO$.
        \EndWhile
        \State Compute $\zeta$ with algorithm \ref{Get_minimal_proper_vertex_height}.
        \If {$\zeta \geq \frac{1}{\mathcal{N}}$ ,}
                \State  \tt All (infinitely many) remaining hemispheres have radius  
                \State smaller than $\zeta$, 

                \State so $(z,\zeta)$ cannot be strictly below any of them. \rm
                \State  Swan's\_cancel\_criterion\_fulfilled := true.
        \Else 
                \State  \tt $\zeta$ becomes the new expected lowest value for $\frac{1}{\mathcal{N}}$: \rm
                \State $E$ := $\frac{1}{\zeta}$.
        \EndIf
\EndWhile
\end{alginc}
}
\end{algorithm}

\newpage
\begin{proposition}
The hemisphere list, as computed by algorithm \emph{\ref{Get_Bianchi_Fundamental_Polyhedron}}, 
determines the Bianchi fundamental polyhedron. This algorithm terminates within finite time.
\end{proposition}
\begin{proof} ${}$

\begin{itemize}
\item The value $\zeta$ is the minimal height of the non-singular vertices of the cell complex 
$\partial B(\alpha_1,\ldots,\alpha_n)$ determined by 
the hemisphere list $\{ S(\alpha_1),\ldots,S(\alpha_n) \}$.

All the hemispheres which are not in the list, have radius smaller than $\frac{1}{\mathcal{N}}$.
 By remark~\ref{expectedLimit}, the inequality $\zeta \geq \frac{1}{\mathcal{N}}$ will become satisfied;
 and then no non-singular vertex of $\partial B(\alpha_1,\ldots,\alpha_n)$ can be strictly below any of them. 
Hence by theorem~\ref{cancelCriterion}, $B(\alpha_1,\ldots,\alpha_n) = B$;
 and we obtain the Bianchi fundamental polyhedron as $B(\alpha_1,\ldots,\alpha_n) \cap D_\infty$.
\item We now consider the run-time.
By \mbox{theorem \ref{finiteNumber}}, the set of hemispheres 
$$\{ S_{\mu,\lambda} \medspace | \medspace S_{\mu,\lambda} \mathrm{ \medspace touches \medspace the \medspace Bianchi \medspace Fundamental \medspace Polyhedron} \}$$
is finite. So, there exists an $S_{\mu,\lambda}$ for which the norm of $\mu$ takes its maximum on this finite set.
The variable $\mathcal{N}$ reaches this maximum for $|\mu|$ after a finite number of steps;
and then Swan's termination criterion is fulfilled.
The latter steps require a finite run-time because of propositions \ref{runtime2} and \ref{runtime3}.
\end{itemize}
\end{proof}

Swan explains furthermore how to obtain an a priori bound for the norm of the $\mu \in \ringO$ occurring for such hemispheres $S_{\mu,\lambda}$. 
But he states that this upper bound for $|\mu|$ is much too large. So instead of the theory behind theorem \ref{finiteNumber}, we use Swan's termination criterion (theorem \ref{cancelCriterion} above) to limit the number of steps in our computations. We then get the following.

\begin{observation} \label{boundsForMu}
We can give bounds for $|\mu|$ in the cases where $K$ is of class number 1 or 2
 (there are nine cases of class number 1 and eighteen cases of class number 2, and we have done the computation for all of them).
They are the following:
$$\begin{cases} \txt{$K$ of class number 1:} \qquad \medspace |\mu| \leq \frac{|\Delta|+1}{2}, \\ \\
\txt{$K$ of class number 2:} \begin{cases} |\mu| \leq 3 |\Delta|, & m \equiv 3 \mod 4, \\
                                        |\mu| \leq (5+\frac{61}{116})|\Delta|, & \mathrm{else},
\end{cases} \end{cases}$$  
where $\Delta$ is the discriminant of $K = \rationals(\sqrt{-m})$, 
 i.e.,
$|\Delta| = \begin{cases} m, & m \equiv 3 \mod 4, \\
4m,  & \mathrm{else.} \end{cases}$
\end{observation}

\vbox{
\begin{remark} \label{expectedLimit}
In algorithm \ref{Get_Bianchi_Fundamental_Polyhedron}, we have chosen the value $E$ by an extra\-polation formula for observation \ref{boundsForMu}.
If this is greater than the  exact bound for $|\mu|$, the algorithm computes additional hemispheres which do not contribute to the Bianchi fundamental polyhedron.
On the other hand, if $E$ is smaller than the  exact bound for $|\mu|$, it will be increased in the outer while loop of the algorithm, until it is sufficiently large. But then, the algorithm performs some preliminary computations of the intersection lines and vertices, which cost additional run-time. Thus our extra\-polation formula is aimed at choosing $E$ slightly greater than the  exact bound for $|\mu|$ we expect.
\end{remark}
}

\begin{algorithm}
\caption{Recording the hemispheres of radius $\frac{1}{\mathcal{N}}$} 
\label{Record_hemispheres}
{
\begin{alginc}
\State \bf Input: \rm The value $\mathcal{N}$, and the hemisphere list (empty by default).
\State \bf Output: \rm The hemisphere list with some hemispheres of radius $\frac{1}{\mathcal{N}}$ added.
\State
\For {$a$ running from $0$ through $\mathcal{N}$ within $\Z$, }
  \For {$b$ in $\Z$ such that $|a +b\omega| = \mathcal{N}$, }
    \State Let $\mu := a +b \omega$.
    \For {all the $\lambda \in \ringO$ with $\frac{\lambda}{\mu}$ in the fundamental rectangle $D_0$,}
        \If{ the pair $(\mu,\lambda)$ is unimodular,}
         \State Let $\mathcal{L}$ be the length of the hemisphere list. 
         \State everywhere\_below := false, \qquad $j := 1$.
         \While{ everywhere\_below = false and $j \leq \mathcal{L}$,}
          \State Let $S_{\beta,\alpha}$ be the $j$'th entry in the hemisphere list;
          \If{ $S_{\mu,\lambda}$ is everywhere below $S_{\beta,\alpha}$,}
                \State everywhere\_below := true.
          \EndIf
          \State Increase $j$ by $1$.
         \EndWhile
         \If{ everywhere\_below = false,}
                \State Record $S_{\mu,\lambda}$ into the hemisphere list.
         \EndIf
        \EndIf
    \EndFor
  \EndFor
\EndFor
\State
\State \tt We recall that the notion ``everywhere below'' has been made precise in definition \ref{everywherebelow}; and that the fundamental rectangle $D_0$ has been specified in notation \ref{fundamentalRectangle}. \rm
\end{alginc}
}
\end{algorithm}

\begin{proposition} \label{runtime2}
Algorithm \emph{\ref{Record_hemispheres}} finds all the hemispheres of radius $\frac{1}{\mathcal{N}}$, on which a 2-cell of the Bianchi fundamental polyhedron can lie. This algorithm terminates within finite time.  
\end{proposition}
\begin{proof} ${}$

\begin{itemize} 
\item Directly from the definition of the hemispheres $S_{\mu,\lambda}$, it follows that the radius is given by $\frac{1}{|\mu|}$. So our algorithm runs through all $\mu$ in question.
By construction of the Bianchi fundamental polyhedron $D$, the hemispheres on which a 2-cell of $D$ lies must have their centre in the fundamental rectangle $D_0$.
By proposition \ref{everywhereBelowHemispheres}, such hemispheres cannot be everywhere below some other hemisphere in the list.
\item Now we consider the run-time of the algorithm. There are finitely many $\mu \in \ringO$ the norm of which takes a given value. 
And for a given $\mu$, there are finitely many $\lambda \in \ringO$ such that 
$\frac{\lambda}{\mu}$ is in the fundamental rectangle $D_0$.  Therefore, this algorithm consists of finite loops and terminates within finite time.
\end{itemize}
\end{proof}

\begin{proposition} \label{runtime3}
Algorithm \emph{\ref{Get_minimal_proper_vertex_height}} finds the minimal height occurring amongst the non-singular vertices of $\partial B(\alpha_1,\ldots,\alpha_n)$. This algorithm erases only such hemispheres from the list, which do not change $\partial B(\alpha_1,\ldots,\alpha_n)$.
It terminates within finite time.
\end{proposition}
\begin{proof} ${}$

\begin{itemize}
\item The heights of the points in $\Hy$ are preserved by the action of the translation group $\Gamma_\infty$, so we only need to consider representatives in the fundamental domain $D_\infty$ for this action. Our algorithm computes the entire cell structure of \mbox{$\partial B(\alpha_1,\ldots,\alpha_n) \cap D_\infty$,} as described in subsection \ref{cellStructure}.
The number of lines to intersect is smaller than the square of the length of the hemisphere list, and thus finite. As a consequence, the minimum of the height has to be taken only on a finite set of intersection points, whence the first claim.
\item  If a cell of $\partial B(\alpha_1,\ldots,\alpha_n)$ lies on a hemisphere, then its vertices are lifts of intersection points. So we can erase the hemispheres which are strictly below some other hemispheres at all the intersection points, without changing $\partial B(\alpha_1,\ldots,\alpha_n)$.
\item Now we consider the run-time.
This algorithm consists of loops running through the hemisphere list, which has finite length. Within one of these loops, there is a loop running through the set of pairs of lines $L(\frac{\alpha}{\beta},\frac{\lambda}{\mu})$. A (far too large) bound for the cardinality of this set is given by the fourth power of the length of the hemisphere list.
\\
The steps performed within these loops are very delimited and easily seen to be of finite run-time.
\end{itemize}
\end{proof}

\begin{algorithm}
\caption{Computing the minimal proper vertex height}
\label{Get_minimal_proper_vertex_height}
{
\begin{alginc}
\State \bf Input: \rm The hemisphere list $\{ S(\alpha_1),\ldots,S(\alpha_n) \}$.
\State \bf Output: \rm The lowest height $\zeta$ of a non-singular vertex of $\partial B(\alpha_1,\ldots,\alpha_n)$. And the hemisphere list with some hemispheres removed which do not touch the Bianchi fundamental polyhedron.
\State
\For {all pairs $S_{\beta,\alpha}$,  $S_{\mu,\lambda}$ in the hemisphere list which touch one another,} 
        \State compute the line $L(\frac{\alpha}{\beta},\frac{\lambda}{\mu})$
                of notation \ref{agreeingLine}.
\EndFor 
\State
\For {each hemisphere $S_{\beta,\alpha}$ in the hemisphere list,} 
   \For{each two lines $L(\frac{\alpha}{\beta},\frac{\lambda}{\mu})$ and $L(\frac{\alpha}{\beta},\frac{\theta}{\tau})$ referring to $\frac{\alpha}{\beta}$,}
        \State Compute the intersection point of  $L(\frac{\alpha}{\beta},\frac{\lambda}{\mu})$ and $L(\frac{\alpha}{\beta},\frac{\theta}{\tau})$, if it exists.
   \EndFor
\EndFor
\State  Drop the intersection points at which $S_{\beta,\alpha}$ is strictly below some hemisphere in the list. 
\State  Erase the hemispheres from our list, for which no intersection points remain. 
\State \tt Now the vertices of \rm $B(\alpha_1,\ldots,\alpha_n) \cap D_\infty$ \tt are the lifts (specified in definition \ref{liftOnHemisphere}) on the appropriate hemispheres of the remaining intersection points.
\State \rm Pick the lowest value $\zeta > 0$ for which $(z,\zeta) \in \Hy$ is the lift on some hemisphere of a remaining intersection point $z$. 
\State Return $\zeta$.
\end{alginc}
}
\end{algorithm}


\subsection{The cell complex and its orbit space} \label{Cell complex}

With the method described in subsection~ \ref{cellStructure}, we obtain a cell structure on the \mbox{boundary} of the Bianchi fundamental polyhedron.
The cells in this structure which touch the cusp~$\infty$ are easily determined:
they are four 2-cells each lying on one of the Euclidean vertical planes bounding the fundamental domain $D_\infty$ for $\Gamma_\infty$ specified in notation~\ref{fundamentalRectangle}; and four 1-cells each lying on one of the intersection lines of these planes.
The other \mbox{2-cells} in this structure lie each on one of the hemispheres determined with our realization of Swan's algorithm.

As the Bianchi fundamental polyhedron is a hyperbolic polyhedron up to some missing cusps, its boundary cells can be oriented as its facets.
Once the cell structure is subdivided until the cells are fixed pointwise by their stabilisers,
 this cell structure with orientation is transported onto the whole hyperbolic space by the action of~$\Gamma$.

\subsection{Computing the vertex stabilisers and identifications} 
Let us state explicitly the $\Gamma$-action on the upper-half space model~$\Hy$,
 in the form in which we will use it rather than in its historical form.
 
\begin{lemma}[Poincar\'e]  \label{operationFormula}
If $\gamma =$\scriptsize$\mat$\normalsize$ \in \mathrm{GL}_2(\C)$, the action of $\gamma$ on $\Hy$ is given by
\\ 
\mbox{$\gamma \cdot (z,\zeta) = (z',\zeta')$,} where
$$ \zeta' = \frac{|\det \gamma|\zeta}{|cz-d|^2 +\zeta^2|c|^2},\qquad z' = \frac{\left(\thinspace\overline{d -cz}\thinspace\right)(az-b) -\zeta^2\bar{c}a}{|cz-d|^2 +\zeta^2|c|^2}.$$
\end{lemma}
From this operation formula, we establish equations and inequalities on the entries of a matrix sending a given point
 $(z,\zeta)$ to another given point $(z',\zeta')$ in $\Hy$.
 We will use them in algorithm~\ref{getVertexIdentifications3mod4}
 to compute such matrices.
For the computation of the vertex stabilisers, we have $(z,\zeta) = (z',\zeta')$,
  which simplifies the below equations and inequalities as well as the pertinent algorithm. 
First, we fix a basis for $\ringO$ as the elements $1$ and 
$$\omega := \begin{cases} 
\sqrt{-m}, &  m \equiv 1 \medspace \mathrm{or} \medspace 2 \mod 4, \\
 -\frac{1}{2}+\frac{1}{2}\sqrt{-m},  &  m \equiv 3 \mod 4. \end{cases}$$
As we have put $m \neq 1$ and $m \neq 3$, the only units in the ring $\ringO$ are $\pm 1$. 
We will use the notations $\left\lceil x \right\rceil := \min \{n \in \Z \suchthat n \geq x \} $  
 and $\lfloor x \rfloor := \max \{n \in \Z \suchthat n \leq x \} $ for $x \in \R$.

\begin{algorithm}
\caption{Computation of the matrices identifying two points in $\Hy$.}
\label{getVertexIdentifications3mod4}
{
\begin{alginc}
\State \bf Input: \rm The points $(z, r)$, $(\zeta,\rho)$ in the interior of $\Hy$, where $z$, $\zeta \in K$ and $r^2$, $\rho^2 \in \rationals$.
\State \bf Output: \rm The set of matrices \scriptsize $\mat$ \normalsize $\in$ SL$_2(\ringO_{-m})$, $m \equiv 3 \mod  4$, with nonzero entry~$c$, sending the first of the input points to the second one.
\State 
\State \tt $c$ will run through $\ringO$ with \rm \mbox{$0 < |c|^2 \leq \frac{1}{r\rho}$.}  
\State Write $c$ in the basis as $ j +k\omega$ , where $j,k \in \Z$. \rm
   \State
   \For{$j$ running from $-\left\lceil\sqrt{\frac{1+\frac{1}{m}}{r\rho}}\thinspace\right\rceil$
                through $\left\lceil\sqrt{\frac{1+\frac{1}{m}}{r\rho}}\thinspace\right\rceil$}
     \State 
     \State $ k_{\mathrm{limit}}^\pm 
                := 2\frac{j}{m+1} \pm 2\frac{\sqrt{\frac{m+1}{r\rho}-j^2m}}{m+1}$. 
     \State
     \For{ $k$ running from $\lfloor k_{\mathrm{limit}}^-\rfloor$
                through    $ \lceil k_{\mathrm{limit}}^+\rceil$}
        \State $c := j + k\omega$;
        \If{$|c|^2 \leq \frac{1}{r\rho}$ and $c$ nonzero,}
            \State Write $cz$ in the basis as $R(cz) + W(cz) \omega$
                with $R(cz), W(cz) \in \rationals$.
            \State \tt   $d$ will run through $\ringO$
                 with $|cz-d|^2 +r^2|c|^2= \frac{r}{\rho}$. \rm
            \State Write $d$ in the basis as $ q +s\omega$, where $q,s \in \Z$. \rm
            \State \tt $s_{\mathrm{limit}}^\pm 
                        := W(cz) \pm 2\sqrt{ \frac{\frac{r}{\rho} -r^2 |c|^2}{m}}$. \rm
            \For{  $s$ running from $ \lfloor s_{\mathrm{limit}}^-\rfloor$
                        through  $ \lceil s_{\mathrm{limit}}^+\rceil$}
              \State
              \State $\Delta :=\frac{r}{\rho} -r^2|c|^2 -m \left( \frac{W(cz)}{2} -\frac{s}{2}\right)^2$;
              \If{ $\Delta$ is a rational square,}
                \State $q_\pm := R(cz) -\frac{W(cz)}{2} +\frac{s}{2} \pm \sqrt{ \Delta}$. 
                \State Do the following for both $q_\pm=q_+$ and $q_\pm =q_-$ if $\Delta \neq 0$.
                \If{ $q_\pm \in \Z$,}
                  \State $d :=  q_\pm + s\omega $;
                  \If{ $|cz -d|^2 +r^2 |c|^2 = \frac{r}{\rho}$ and $(c,d)$ unimodular, }
                    \State $ a := \frac{\rho}{r}\overline{d} -\frac{\rho}{r}\overline{cz} -c\zeta$.
                    \If{$a$ is in the ring of integers,}
                        \State \tt $b$ is determined by the determinant $1$: \rm
                        \State $ b := \frac{ad -1}{ c}$.
                        \If {$b$ is in the ring of integers,}
                                \State Check that \scriptsize $\mat$\normalsize$ \cdot (z,r) = (\zeta,\rho)$.
                                \State Return \scriptsize$ \mat $\normalsize.
                        \EndIf
                    \EndIf
                  \EndIf
                \EndIf
              \EndIf
        \EndFor
    \EndIf
  \EndFor
\EndFor
\end{alginc}
}
\end{algorithm}

\begin{lemma} \label{jAndk}
Let $m \equiv 3 \mod 4.$ Let \scriptsize $\mat$\normalsize$ \in \mathrm{SL}_2(\ringO)$
  be a matrix sending $(z,r)$ to \mbox{$(\zeta,\rho) \in \Hy$.}
 Write $c$ in the basis as $ j +k\omega$ , where $j,k \in \Z$. 
Then \mbox{$|c|^2 \leq \frac{1}{r\rho}$}, $\medspace$
$|j| \leq \sqrt{\frac{1+\frac{1}{m}}{r\rho}}$ and
 $$\frac{2j}{m+1} - 2\frac{\sqrt{\frac{m+1}{r\rho}-j^2m}}{m+1} \leq k \leq \frac{2j}{m+1} + 2\frac{\sqrt{\frac{m+1}{r\rho}-j^2m}}{m+1}.$$  
\end{lemma}

\begin{proof}
From the operation equation \scriptsize $\mat$\normalsize$\cdot (z,r) = (\zeta,\rho)$, we deduce 
\mbox{$|cz-d|^2 +r^2|c|^2= \frac{r}{\rho}$} and conclude \mbox{$ r^2|c|^2 \leq \frac{r}{\rho} $,} whence the first inequality.
We insert $|c|^2 = \left(j -\frac{k}{2}\right)^2 +m\left(\frac{k}{2}\right)^2$  
                                    $= j^2 +\frac{m+1}{4}k^2 -jk $ \thinspace into it, 
and obtain 
$$0 \geq k^2 -\frac{4j}{m+1}k + \frac{4}{m+1}\left(j^2 -\frac{1}{r\rho}\right) =: f(k).$$
We observe that $f(k)$ is a quadratic function in $k\in \Z \subset \R$, taking its values exclusively in $\R$.
Hence its graph has the shape of a parabola,  and the negative values of $f(k)$ appear exactly on the interval where $k$ is between its two zeroes,
\begin{center} $ k_\pm = \frac{2j}{m+1} \pm 2\frac{\sqrt{\Delta}}{m+1},$ \qquad
where $\Delta = \frac{m+1}{r\rho}-j^2m$. \end{center}
This implies the third and fourth claimed inequalities.
As $k$ is a real number, $\Delta$ must be non-negative in order that $f(k)$ be non-positive.
Hence $j^2 \leq \frac{1+\frac{1}{m}}{r\rho}$, which gives the second claimed inequality.
\end{proof}

\newpage
\begin{lemma} \label{sAndq}
Under the assumptions of lemma \emph{\ref{jAndk}}, write $d$ in the basis as $ q +s\omega$, where $q,s \in \Z$. Write $cz$ in the basis as $R(cz) + W(cz) \omega$, 
where $R(cz), W(cz) \in \rationals$.
Then 
$ W(cz) - 2\sqrt{ \frac{\frac{r}{\rho} -r^2 |c|^2}{m}} \leq s \leq
  W(cz) + 2\sqrt{ \frac{\frac{r}{\rho} -r^2 |c|^2}{m}}$, and
$$ q = R(cz) -\frac{W(cz)}{2} +\frac{s}{2} 
\pm \sqrt{\frac{r}{\rho} -r^2 |c|^2 
-m \left( \frac{W(cz)}{2} -\frac{s}{2}\right)^2}.$$
\end{lemma}
\begin{proof}
Recall that $\omega = -\frac{1}{2}+\frac{1}{2}\sqrt{-m}$,
so $\overline{q+s\omega} = q - \frac{s}{2} -\frac{s}{2}\sqrt{-m}$.
The operation equation yields $|cz-d|^2 +r^2|c|^2= \frac{r}{\rho}$. 
From this, we derive 
$$\begin{array}{rl}
 \frac{r}{\rho} -r^2|c|^2 &
= \left(cz -(q+s\omega)\right)\left(\overline{cz} -(q - \frac{s}{2} -\frac{s}{2}\sqrt{-m})\right) \\
 & = \left(\Real(cz) -q +\frac{s}{2}\right)^2 +\left(\Imag(cz) -\frac{s}{2}\sqrt{m}\right)^2 \\
 & = \Real(cz)^2 +q^2 -qs +\frac{s^2}{4} -2 \Real(cz)q +\Real(cz)s +\left(\Imag(cz) -\frac{s}{2}\sqrt{m}\right)^2. 
\end{array}$$
We solve for $q$,
$$
q^2 +\left(-2 \Real(cz) -s \right)q +\left(\Real(cz) +\frac{s}{2}\right)^2 +\left(\Imag(cz) -\frac{s}{2}\sqrt{m}\right)^2 -\frac{r}{\rho} +r^2|c|^2 = 0
$$
and find
\begin{center} $ q_\pm = \Real(cz) +\frac{s}{2} \pm \sqrt{\Delta},$ \qquad
where $\Delta =\frac{r}{\rho} -r^2|c|^2 -\left(\Imag(cz) -\frac{s}{2}\sqrt{m}\right)^2$. \end{center}
We express this as
\begin{center} $ q_\pm = R(cz)  -\frac{W(cz)}{2} +\frac{s}{2} \pm \sqrt{\Delta},$ \qquad
where $\Delta =\frac{r}{\rho} -r^2|c|^2 -m \left( \frac{W(cz)}{2} -\frac{s}{2}\right)^2$, \end{center}
which is the claimed equation. The condition that $q$ must be a rational integer implies $\Delta \geq 0$, which can be rewritten in the claimed inequalities. 
\end{proof}

We further state a simple inequality in order to prove that algorithm~\ref{getVertexIdentifications3mod4}  terminates in finite time.

\begin{lemma} \label{simpleInequation}
Let $K = \rationals(\sqrt{-m})$ with $m \neq 3$. Let $c,z \in K$.
Write their product $cz$ in the $\rationals$-basis $\{ 1, \omega \}$ for $K$ as
\mbox{$R(cz) + W(cz) \omega$.} 
 Then the inequality $|W(cz)| \leq |c|\cdot |z|$ holds.
\end{lemma}
\begin{proof}
Let $x + y \omega \in K$  with $x, y \in \rationals$. Our first step is to show that
\mbox{$|y| \leq |x + y \omega|$.}
Consider the case $m \equiv 1 \medspace \mathrm{or} \medspace 2 \mod 4$.
Then $$|x + y \omega| = \sqrt{ x^2 +my^2} \geq \sqrt{m}|y| \geq |y|,$$ and we have shown our claim.
Else consider the case $m \equiv 3 \mod 4$. Then,
$$|x + y \omega| = \sqrt{ (x +\omega y)(x+\overline{\omega}y)} 
= \sqrt{ \left(x^2 -2x\frac{y}{2} +\frac{y^2}{4}\right) +\frac{m}{4}y^2} \geq \frac{\sqrt{m}}{2}|y|,$$
and our claim follows for $m>3$. 
Now we have shown that $|W(cz)| \leq |cz|$; 
and we use some embedding of $K$ into $\C$ to verify the equation 
$|cz| = |c|\cdot|z|$.
\end{proof}

\begin{proposition}
Let $m \equiv 3 \mod 4.$  Then algorithm \emph{\ref{getVertexIdentifications3mod4}} gives all the matrices 
\mbox{\scriptsize \mbox{$\mat$\normalsize$ \in \mathrm{SL}_2(\ringO)$} \normalsize}
 with $c \neq 0$, sending $(z,r)$ to $(\zeta,\rho) \in \Hy$. It terminates in finite time.
\end{proposition}
\begin{proof} ${}$
\begin{itemize}
\item The first claim is easily established using the bounds and formulae stated in lemmata \ref{jAndk} and \ref{sAndq}.
\item 
Now we consider the run-time. This algorithm consists of three loops the limits of which are at most linear expressions in $\frac{1}{\sqrt{r\rho}}$.
 For $s_{\mathrm{limit}}^\pm$, we use lemma \ref{simpleInequation} and $ r^2|c|^2 \leq \frac{r}{\rho} $ to see this (we get a factor $|z|$ here, which we can neglect).
\end{itemize}
\end{proof}

Finally, it should be said that the scope of computations one can do with geometric models for the Bianchi groups
 does not stop once the integral homology of the full group is known. 
There is further interest in homology with twisted coefficients, congruence subgroups and modular forms 
(see for instance \cite{Sengun}, \cite{SengunTurkelli}).
Currently, Page~\cite{AurelPage}
 is working on optimizing algorithms in order to obtain more cell complexes for Bianchi groups and other Kleinian groups.

\subsection*{Acknowledgements} The author would like to thank Bill Allombert (PARI/GP Development Headquarters)
 and Philippe Elbaz-Vincent (UJF Grenoble) for invaluable help with the development of \textit{Bianchi.gp}.
He is grateful to the Weizmann Institute of Science for providing him its high perfomance computation clusters
 in order to establish the cell complexes database;
and to Stephen S. Gelbart and Graham Ellis for support and encouragement.

\bibliographystyle{amsalpha}
\providecommand{\bysame}{\leavevmode\hbox to3em{\hrulefill}\thinspace}
\providecommand{\MR}{\relax\ifhmode\unskip\space\fi MR }
\providecommand{\MRhref}[2]{%
  \href{http://www.ams.org/mathscinet-getitem?mr=#1}{#2}
}
\providecommand{\href}[2]{#2}

\end{document}